\documentclass[a4paper,leqno]{article}
\usepackage[dvips]{graphicx}
\usepackage{amssymb}
\usepackage{amsmath}
\usepackage{mathrsfs}
\usepackage{color}
\usepackage{theorem}
\usepackage{comment}
\usepackage{url}
\usepackage{ascmac}
\newcommand{\pref}[1]{(\ref{#1})}
\newcommand{\be}{\begin{equation}}
\newcommand{\ee}{\end{equation}}
\newcommand{\qed}{{\unskip\nobreak\hfil\penalty50\quad\null\nobreak\hfil
	$\square$\parfillskip0pt\finalhyphendemerits0\par\medskip}}
\newcommand{\zume}{\!\!\!}
\newcommand{\veca}{\mbox{\boldmath $ a $}}
\newcommand{\vecb}{\mbox{\boldmath $ b $}}
\newcommand{\vecc}{\mbox{\boldmath $ c $}}

\newcommand{\vecf}{\mbox{\boldmath $ f $}}

\newcommand{\vecu}{\mbox{\boldmath $ u $}}
\newcommand{\vecv}{\mbox{\boldmath $ v $}}
\newcommand{\vecw}{\mbox{\boldmath $ w $}}

\newcommand{\vecphi}{\mbox{\boldmath $ \phi $}}
\newcommand{\vecpsi}{\mbox{\boldmath $ \psi $}}

\newcommand{\vectau}{\mbox{\boldmath $ \tau $}}

\newcommand{\smallvecf}{\mbox{\scriptsize \boldmath $ f $}}

\renewcommand{\epsilon}{\varepsilon}

\newtheorem{thm}{Theorem}[section]
\newtheorem{prop}{Proposition}[section]
\newtheorem{lem}{Lemma}[section]
\newtheorem{cor}{Corollary}[section]
\newtheorem{eg}{Example}[section]

\theorembodyfont{\rmfamily}
\newtheorem{rem}{Remark}[section]
\newtheorem{proof}{\normalfont\itshape Proof.}

\title{Decomposition of generalized O'Hara's energies}
\author{
Aya Ishizeki, Chiba University
\and
Takeyuki Nagasawa, Saitama University
}
\date{\today}
\begin{document}
\maketitle
\begin{abstract}
O'Hara introduced several functionals as knot energies.
One of them is the M\"{o}bius energy.
We know its M\"{o}bius invariance from Doyle-Schramm's cosine formula.
It is also known that the M\"{o}bius energy was decomposed into three components keeping the M\"{o}bius invariance.
The first component of decomposition represents the extent of bending of the curves or knots, while the second one indicates the extent of twisting.
The third one is an absolute constant.
In this paper,
we show a similar decomposition for generalized O'Hara energies.
On the way to derive it,
we obtain an analogue of the cosine formula for the generalized O'Hara energy.  
Furthermore, using decomposition,
the first and second variational formulae are derived.
\\
keywords:
O'Hara energy \and M\"{o}bius energy \and knot energy \and decomposition of energy \and variational formula\\
subclass:
53A04 \and 58J70 \and 49Q10
\end{abstract}
\section{Introduction}
\par
Let $ \vecf $ be a knot with length $ \mathcal{L} $ which is parameteried by arc-length. 
In \cite{OH},
the functional
\[
	\mathcal{E} ( \vecf )
	=
	\iint_{ ( \mathbb{R} / \mathcal{L} \mathbb{Z} )^2 }
	\left(
	\frac 1 { \| \vecf ( s_1 ) - \vecf ( s_2 ) \|_{ \mathbb{R}^3 }^\alpha }
	-
	\frac 1 { \mathscr{D} ( \vecf ( s_1 ) , \vecf ( s_2 ) )^\alpha }
	\right)^p
	d s_1 d s_2
\]
was introduced by O'Hara as an energy of knots, where 
$ \mathscr{D} ( \vecf ( s_1 ) , \vecf ( s_2 ) ) $ is distance between two points $ \vecf ( s_1 ) $ and $ \vecf ( s_2 ) $ on the curve $ \vecf $.
$ \alpha $ and $ p $ are positive constants.  
We call it O'Hara's $ ( \alpha , p ) $ energy. 
Though a knot is defined as a closed curve in $ \mathbb{R}^3 $ without self-intersections, 
the above energy can be defined for curves in $ \mathbb{R}^n $. 
Therefore here we consider $ \vecf $ as a closed curve in $ \mathbb{R}^n $ without self-intersections.
We denote the total length by $ \mathcal{L} $. 
$ \| \cdot \|_{ \mathbb{R}^n } $ denotes Euclidean norm in $ \mathbb{R}^n $. 
\par
Let $ \Phi $ be a function from $ \mathbb{R}_+ = \{ x \in \mathbb{R} \, | \, x > 0 \} $ to itself. We consider the energy
\[
	\mathcal{E}_\Phi ( \vecf )
	=
	\iint_{ ( \mathbb{R} / \mathcal{L} \mathbb{Z} )^2 }
	\left(
	\frac 1 { \Phi ( \| \vecf ( s_1 ) - \vecf ( s_2 ) \|_{ \mathbb{R}^n } ) }
	-
	\frac 1 { \Phi ( \mathscr{D} ( \vecf ( s_1 ) , \vecf ( s_2 ) ) ) }
	\right)
	d s_1 d s_2
\]
for closed curves in $ \mathbb{R}^n $. 
Taking the fact that the above energy is O'Hara's $ ( \alpha , 1 ) $ energy if $ \Phi (x) = x^\alpha $ into consideration, it is natural to call it a generalized O'Hara energy. 
Since non-negativity of this energy density is necessary, we suppose that 
\begin{itemize}
\item[{\rm (A.1)}] $ \Phi $ is monotonically increasing. 
\end{itemize}
\par
In what follows,
we use the notation $ \Delta u $ to mean $ u( s_1 ) - u ( s_2 ) $ for a function on $ \mathbb{R} / \mathcal{L} \mathbb{Z} $.
\par
In case of $ \Phi (x) = x^2 $, it holds that 
\[
	\mathcal{E}_{ x^2 } ( \vecf )
	=
	\iint_{ ( \mathbb{R} / \mathcal{L} \mathbb{R} )^2 }
	\frac { 1 - \cos \varphi ( s_1 , s_2 ) }
	{ \| \Delta \vecf \|_{ \mathbb{R}^n }^2 }
	\, d s_1 d s_2
	+ 4
\]
by using Doyle-Schramm's formula, where $ \varphi $ is a conformal angle (see \cite{KS}).
This is called the cosine formula. 
Since the integrand is M\"{o}bius invariant and so is this energy itself. 
\begin{rem}
The invariance of $ \mathcal{E}_{ x^2 } $ energy is firstly proved by Freedman-He-Wang \cite{FHW} in other way. 
\end{rem}
\par
We showed that $ \mathcal{E}_{ x^2 } $ may be decomposed like
\begin{align*}
	\mathcal{E}_{ x^2 } ( \vecf )
	= & \
	\mathcal{E}_{ x^2 , 1 } ( \vecf )
	+
	\mathcal{E}_{ x^2 , 2 } ( \vecf )
	+
	4
	,
	\\
	\mathcal{E}_{ x^2 , 1 } ( \vecf )
	= & \
	\iint_{ ( \mathbb{R} / \mathcal{L} \mathbb{R} )^2 }
	\frac { \| \Delta \vectau \|_{ \mathbb{R}^n }^2 } { 2 \| \Delta \vecf \|_{ \mathbb{R}^n }^2 } \, d s_1 d s_2
	,
	\\
	\mathcal{E}_{ x^2 , 2 } ( \vecf )
	= & \
	\iint_{ ( \mathbb{R} / \mathcal{L} \mathbb{R} )^2 }
	\frac 2 { \| \Delta \vecf \|_{ \mathbb{R}^n }^2 }
	\left\langle
	\vectau ( s_1 ) \wedge \frac { \Delta \vecf } { \| \Delta \vecf \|_{ \mathbb{R}^n } }
	,
	\vectau ( s_2 ) \wedge \frac { \Delta \vecf } { \| \Delta \vecf \|_{ \mathbb{R}^n } }
	\right\rangle_{ {\bigwedge}^2 \mathbb{R}^n }
	\, d s_1 d s_2
\end{align*}
and $ \mathcal{E}_{ x^2 , 1 } $,
$ \mathcal{E}_{ x^2 , 2 } $ are M\"{o}bius invariant (\cite{IshizekiNagasawaI,IshizekiNagasawaIII}). 
Here $ \vectau = \vecf ^\prime $ is the unit tangent vector.
$ \wedge $ is the exterior product between two vectors on $ \mathbb{R}^n $.
$ \langle \cdot , \cdot \rangle_{ {\bigwedge}^2 \mathbb{R}^n } $ is the inner product on the space $ {\bigwedge}^2 \mathbb{R}^n $ of 2-vectors.
Note that the existence of $ \vectau $ almost everywhere follows from the finiteness of energy by Blatt \cite{Blatt}.
The first one $ \mathcal{E}_{ x^2 , 1 } $ represents how bent curves or knots are and the second one $ \mathcal{E}_{ x^2 , 2 } $ does how twisted they are. 
Therefore it is natural to consider whether the generalized O'Hara's energy may be decomposed or not. 
It seems that the decomposition of the M\"{o}bius energy can be gained by decomposing the cosine formula, however our proof (\cite{IshizekiNagasawaI}) does not depend on that formula. 
This suggests that the generalized O'Hara's energy may be decomposed in a similar way as the M\"{o}bius energy. 
We do not consider M\"{o}bius invariance in this case. 
\par
We suppose further conditions for $ \Phi $ as follows. 
\begin{itemize}
\item[{\rm (A.2)}]
	$ \displaystyle{ \int_x^\infty \frac { dt } { \Phi (t) } < \infty } $ for $ x > 0 $, 
\item[{\rm (A.3)}]
	Functional space $ W_\Phi $ is defined by
	\[
		W_\Phi
		=
		\left\{ \vecu \in W^{1,2} ( \mathbb{R} / \mathcal{L} \mathbb{Z} )
		\, \left| \,
		\iint_{ ( \mathbb{R} / \mathcal{L} \mathbb{Z} )^2 }
		\frac { \| \Delta \vecu^\prime \|_{ \mathbb{R}^n }^2 }
		{ \Phi ( \mathrm{dist}_{ \mathbb{R} / \mathcal{L} \mathbb{Z} } ( s_1 , s_2 ) ) }
		\, d s_1 d s_2
		< 
		\infty
		\right. \right\}. 
	\]
	If $ \mathcal{E}_\Phi ( \vecf ) < \infty $, then 
	\[
		\vecf \in W_\Phi \cap W^{ 1, \infty } ( \mathbb{R} / \mathcal{L} \mathbb{Z} ) ,
		\quad
		\vecf \mbox{ is bi-Lipschitz}, 
	\]
	\item[{\rm (A.4)}]
	Set
	\[
		\Lambda (x) = - \frac 1x \int_x^\infty \frac { dt } { \Phi (t) }. 
	\]
	If $ \vecf  \in W_\Phi \cap W^{ 1, \infty } ( \mathbb{R} / \mathcal{L} \mathbb{Z} )  $ is bi-Lipschitz, and 
	$ \| \vecf^\prime \|_{ \mathbb{R}^n } \equiv 1 $ (a.e.), then it holds that
	\begin{align*}
		& ( \ast ) \quad
		\lim_{ \epsilon \to + 0 }
		\epsilon
		\int_{ \mathbb{R} / \mathcal{L} \mathbb{Z} }
		\left( \Lambda ( \| \vecf ( s_1 ) - \vecf ( s_1 + \epsilon ) \|_{ \mathbb{R}^n } )
		- \Lambda ( \epsilon ) \right)
		d s_1
		= 0 ,
		\\
		& ( \dagger ) \quad
		\lim_{ \epsilon \to + 0 }
		\int_{ \mathbb{R} / \mathcal{L} \mathbb{Z} }
		\Lambda ( \| \vecf ( s_1 ) - \vecf ( s_1 + \epsilon ) \|_{ \mathbb{R}^n } )
		\int_{ s_1 }^{ s_1 + \epsilon }
		\| \vecf^\prime ( s_1 ) - \vecf^\prime ( s_2 ) \|_{ \mathbb{R}^n }^2
		 d s_2 d s_1
		 = 0. 
	\end{align*}
\item[{\rm (A.5)}]
	\begin{itemize}
	\item[{\rm (a)}]
	For any $ \lambda \in ( 0,1 ) $ and any $ x \in \left( 0 , \frac { \mathcal{L} } 2 \right] $, there exists a constant $ C ( \lambda , \mathcal{L} ) > 0 $ such that $ \Phi ( \lambda x ) \geqq C( \lambda , \mathcal{L} ) \Phi ( x ) $, 
	\item[{\rm (b)}]
	$ \displaystyle{ \inf_{ x \in \left( 0 , \frac { \mathcal{L} } 2 \right] }
\left( \frac 1 { \Phi (x) } + \Lambda (x) \right) \geqq 0 } $.
	\end{itemize}
\end{itemize}
\par
In \S~\ref{Sufficient conditions} we refer to sufficient conditions of $ \Phi $ for these conditions to hold. 
Using the sufficient conditions, we show that (A.1)--(A.5) hold if $ \alpha \in [ 2 , 3 ) $ in case of $ \Phi (x) = x^\alpha $. 
The article \cite{OH3} says that an inequality $ \alpha \geqq 2 $ is a condition of self-repulsiveness of  $ \mathcal{E}_{ x^\alpha } $. 
On the other hand, $ \alpha < 3 $ is a condition of $ \mathcal{E}_{ x^\alpha } ( \vecf ) < \infty $ for any smooth closed curves without self-intersections. 
Taking them into consideration, $ \mathcal{E}_{ x^\alpha } $ is well-defined as energy of knots if $ \alpha \in [ 2 , 3 ) $. 
\par
Let us describe our main theorem. 
\begin{thm}
We suppose {\rm (A.1)--(A.5)}. Set
\begin{align*}
	\mathcal{E}_{ \Phi , 1 } ( \vecf )
	= & \
	\iint_{ ( \mathbb{R} / \mathcal{L} \mathbb{Z} )^2 }
	\frac { \| \Delta \vectau \|_{ \mathbb{R}^n }^2 }
	{ 2 \Phi ( \| \Delta \vecf \|_{ \mathbb{R}^n } ) }
	\, d s_1 d s_2
	,
	\\
	\mathcal{E}_{ \Phi , 2 } ( \vecf )
	= & \
	\iint_{ ( \mathbb{R} / \mathcal{L} \mathbb{Z} )^2 }
	\left(
	\frac 1 { \Phi ( \| \Delta \vecf \|_{ \mathbb{R}^n } ) }
	-
	\Lambda ( \| \Delta \vecf \|_{ \mathbb{R}^n } )
	\right)
	\\
	& \quad
	\times
	\left\langle
	\vectau ( s_1 ) \wedge \frac { \Delta \vecf } { \| \Delta \vecf \|_{ \mathbb{R}^n } }
	,
	\vectau ( s_2 ) \wedge \frac { \Delta \vecf } { \| \Delta \vecf \|_{ \mathbb{R}^n } }
	\right\rangle_{ {\bigwedge}^2 \mathbb{R}^n }
	d s_1 d s_2
	.
\end{align*}
If $ \mathcal{E}_\Phi ( \vecf ) < \infty $, then the integrals of 
$ \mathcal{E}_{ \Phi , 1 } ( \vecf ) $ and $ \mathcal{E}_{ \Phi , 2 } ( \vecf ) $ are absolutely convergent. Furthermore, it follows that
\[
	\mathcal{E}_\Phi ( \vecf )
	=
	\mathcal{E}_{ \Phi , 1 } ( \vecf )
	+
	\mathcal{E}_{ \Phi , 2 } ( \vecf )
	+
	2 \mathcal{L} \int_{ \frac { \mathcal{L} } 2 }^\infty \frac { dx } { \Phi (x) }. 
\]
\label{Main Theorem}
\end{thm}
\par
This decomposition theorem corresponds to that gained in \cite{IshizekiNagasawaI} in case of $ \Phi (x) = x^2 $. 
We give a proof in \S~\ref{Proof}. 
\par
In \cite{IshizekiNagasawaII} we gave a first and second variational formula in use of this theorem and we had estimates in some  functional spaces. 
In \S~\ref{Variational formulae}, 
we derive those formulae for the generalized O'Hara's energy. 
\section{Proof of the Main theorem}
\label{Proof}
\par
We denote energy density of $ \mathcal{E}_\Phi ( \vecf ) $,
$ \mathcal{E}_{ \Phi , 1 } ( \vecf ) $,
$ \mathcal{E}_{ \Phi , 2 } ( \vecf ) $ as $ \mathscr{M}_\Phi ( \vecf ) $,
$ \mathscr{M}_{ \Phi , 1 } ( \vecf ) $,
$ \mathscr{M}_{ \Phi , 2 } ( \vecf ) $, respectively. 
Since $ \mathcal{E}_\Phi $ is non-negative from the condition (A.1), it holds that 
\[
	\mathcal{E}_\Phi ( \vecf )
	=
	\lim_{ \epsilon \to + 0 }
	\int_{ \mathbb{R} / \mathcal{L} \mathbb{Z} }
	\left(
	\int_{ s_1 - \frac { \mathcal{L} } 2 + \epsilon }^{ s_1 - \epsilon }
	+
	\int_{ s_1 + \epsilon }^{ s_1 + \frac { \mathcal{L} } 2 - \epsilon }
	\right)
	\mathscr{M}_\Phi ( \vecf ) \,
	d s_2 d s_1. 
\]
We deform $ \mathscr{M}_\Phi ( \vecf ) $ which is a function with regard to $ ( s_1 , s_2 ) $. 
Assuming the condition (A.2), we set 
\[
	\Psi (x) = - \int \left( \int_x^\infty \frac { dt } { \Phi (t) } \right) dx. 
\]
An integral with regard to $ dx $ is an indefinite integral. 
\begin{lem}
If $ s_1 $ and $ s_2 $ satisfy $ 0 < | s_1 - s_2 | < \frac { \mathcal{L} } 2 $, then it follows that 
\begin{align*}
	\mathscr{M}_\Phi ( \vecf )
	= & \
	\mathscr{M}_{ \Phi , 1 } ( \vecf )
	+
	\mathscr{M}_{ \Phi , 2 } ( \vecf )
	\\
	& \quad
	+ \,
	\frac { \partial^2 } { \partial s_1 \partial s_2 }
	\left(
	\Psi ( \mathscr{D} ( \vecf ( s_1 ) , \vecf ( s_2 ) )
	-
	\Psi ( \| \vecf ( s_1 ) - \vecf ( s_2 ) \|_{ \mathbb{R}^n } )
	\right). 
\end{align*}
\end{lem}
\begin{proof}	
When $ | s_1 - s_2 | \leqq \frac { \mathcal{L} } 2 $, $ \mathscr{D} ( \vecf ( s_1 ) , \vecf ( s_2 ) )^2 = ( s_1 - s_2 )^2 $ holds. 
Therefore if $ 0 < | s_1 - s_2 | < \frac { \mathcal{L} } 2 $, it follows that 
\[
	\frac { \partial \mathscr{D} ( \vecf ( s_1 ) , \vecf ( s_2 ) ) } { \partial s_1 }
	=
	\frac 1 { 2 \mathscr{D} ( \vecf ( s_1 ) , \vecf ( s_2 ) ) }
	\frac { \partial \mathscr{D} ( \vecf ( s_1 ) , \vecf ( s_2 ) )^2 } { \partial s_1 }
	=
	\frac { s_1 - s_2 } { \mathscr{D} ( \vecf ( s_1 ) , \vecf ( s_2 ) ) }
	,
\]
\[
	\frac { \partial \mathscr{D} ( \vecf ( s_1 ) , \vecf ( s_2 ) ) } { \partial s_2 }
	=
	\frac 1 { 2 \mathscr{D} ( \vecf ( s_1 ) , \vecf ( s_2 ) ) }
	\frac { \partial \mathscr{D} ( \vecf ( s_1 ) , \vecf ( s_2 ) )^2 } { \partial s_2 }
	=
	\frac { s_2 - s_1 } { \mathscr{D} ( \vecf ( s_1 ) , \vecf ( s_2 ) ) }
	,
\]
\begin{align*}
	\frac { \partial^2 \mathscr{D} ( \vecf ( s_1 ) , \vecf ( s_2 ) ) } { \partial s_1 \partial s_2 }
	= & \
	\frac \partial { \partial s_1 }
	\frac { s_2 - s_1 } { \mathscr{D} ( \vecf ( s_1 ) , \vecf ( s_2 ) ) }
	\\
	= & \
	- \frac 1 { \mathscr{D} ( \vecf ( s_1 ) , \vecf ( s_2 ) ) }
	-
	\frac { s_2 - s_1 } { \mathscr{D} ( \vecf ( s_1 ) , \vecf ( s_2 ) )^2 }
	\frac { \partial \mathscr{D} ( \vecf ( s_1 ) , \vecf ( s_2 ) ) } { \partial s_1 }
	\\
	= & \
	- \frac 1 { \mathscr{D} ( \vecf ( s_1 ) , \vecf ( s_2 ) ) }
	+
	\frac { ( s_1 - s_2 )^2 } { \mathscr{D} ( \vecf ( s_1 ) , \vecf ( s_2 ) )^3 }
	\\
	= & \
	0. 
\end{align*}
Calculating straightforward, we obtain 
\begin{align*}
	&
	\frac { \partial^2 } { \partial s_1 \partial s_2 } \Psi ( \mathscr{D} ( \vecf ( s_1 ) , \vecf ( s_2 ) )
	\\
	& \quad
	=
	\frac \partial { \partial s_1 } \left( \Psi^\prime ( \mathscr{D} ( \vecf ( s_1 ) , \vecf ( s_2 ) )
	\frac { \partial \mathscr{D} ( \vecf ( s_1 ) , \vecf ( s_2 ) ) } { \partial s_2 }
	\right)
	\\
	& \quad
	=
	\Psi^{ \prime \prime } ( \mathscr{D} ( \vecf ( s_1 ) , \vecf ( s_2 ) )
	\frac { \partial \mathscr{D} ( \vecf ( s_1 ) , \vecf ( s_2 ) ) } { \partial s_1 } 
	\frac { \partial \mathscr{D} ( \vecf ( s_1 ) , \vecf ( s_2 ) ) } { \partial s_2 }
	\\
	& \quad \qquad
	+ \,
	\Psi^\prime ( \mathscr{D} ( \vecf ( s_1 ) , \vecf ( s_2 ) )
	\frac { \partial^2 \mathscr{D} ( \vecf ( s_1 ) , \vecf ( s_2 ) ) } { \partial s_1 \partial s_2 }
	\\
	& \quad
	=
	- \frac 1 { \Phi ( \mathscr{D} ( \vecf ( s_1 ) , \vecf ( s_2 ) ) ) }
\end{align*}
and therefore we show that 
\begin{align*}
	\mathscr{M}_\Phi ( \vecf )
	= & \
	\frac 1 { \Phi ( \| \vecf ( s_1 ) - \vecf ( s_2 ) \|_{ \mathbb{R}^n } ) }
	-
	\frac 1 { \Phi ( \mathscr{D} ( \vecf ( s_1 ) , \vecf ( s_2 ) ) ) }
	\\
	= & \
	\frac 1 { \Phi ( \| \vecf ( s_1 ) - \vecf ( s_2 ) \|_{ \mathbb{R}^n } ) }
	+
	\frac { \partial^2 } { \partial s_1 \partial s_2 } \Psi ( \| \vecf ( s_1 ) - \vecf ( s_2 ) \|_{ \mathbb{R}^n } )
	\\
	& \quad
	+ \,
	\frac { \partial^2 } { \partial s_1 \partial s_2 }
	\left(
	\Psi ( \mathscr{D} ( \vecf ( s_1 ) , \vecf ( s_2 ) )
	-
	\Psi ( \| \vecf ( s_1 ) - \vecf ( s_2 ) \|_{ \mathbb{R}^n } )
	\right). 
\end{align*}
Using the above calculation, we obtain 
\begin{align*}
	\frac \partial { \partial s_1 } \| \vecf ( s_1 ) - \vecf ( s_2 ) \|_{ \mathbb{R}^n }
	= & \
	\frac { \vectau ( s_1 ) \cdot ( \vecf ( s_1 ) - \vecf ( s_2 ) ) }
	{ \| \vecf ( s_1 ) - \vecf ( s_2 ) \|_{ \mathbb{R}^n } }
	=
	\vectau ( s_1 ) \cdot
	\frac { \Delta \vecf } { \| \Delta \vecf \|_{ \mathbb{R}^n } }
	,
	\\
	\frac \partial { \partial s_2 } \| \vecf ( s_1 ) - \vecf ( s_2 ) \|_{ \mathbb{R}^n }
	= & \
	-
	\frac { \vectau ( s_2 ) \cdot ( \vecf ( s_1 ) - \vecf ( s_2 ) ) }
	{ \| \vecf ( s_1 ) - \vecf ( s_2 ) \|_{ \mathbb{R}^n } }
	=
	-
	\vectau ( s_2 ) \cdot
	\frac { \Delta \vecf } { \| \Delta \vecf \|_{ \mathbb{R}^n } }
	,
\end{align*}
\begin{align*}
	&
	\frac { \partial^2 } { \partial s_1 \partial s_2 }
	\| \vecf ( s_1 ) - \vecf ( s_2 ) \|_{ \mathbb{R}^n }
	\\
	& \quad
	=
	- \frac { \vectau ( s_1 ) \cdot \vectau ( s_2 ) }
	{ \| \vecf ( s_1 ) - \vecf ( s_2 ) \|_{ \mathbb{R}^n } }
	+
	\frac {
	\{ \vectau ( s_1 ) \cdot ( \vecf ( s_1 ) - \vecf ( s_2 ) ) \}
	\cdot
	\{ \vectau ( s_2 ) \cdot ( \vecf ( s_1 ) - \vecf ( s_2 ) ) \}
	}
	{ \| \vecf ( s_1 ) - \vecf ( s_2 ) \|_{ \mathbb{R}^n }^3 }
	\\
	& \quad
	=
	- \frac 1 { \| \Delta \vecf \|_{ \mathbb{R}^n } }
	\left\{
	\vectau ( s_1 ) \cdot \vectau ( s_2 )
	-
	\left(
	\vectau ( s_1 ) \cdot \frac { \Delta \vecf } { \| \Delta \vecf \|_{ \mathbb{R}^n } }
	\right)
	\left(
	\vectau ( s_2 ) \cdot \frac { \Delta \vecf } { \| \Delta \vecf \|_{ \mathbb{R}^n } }
	\right)
	\right\}
	\\
	& \quad
	=
	- \frac 1 { \| \Delta \vecf \|_{ \mathbb{R}^n } }
	\left\langle
	\vectau ( s_1 ) \wedge \frac { \Delta \vecf } { \| \Delta \vecf \|_{ \mathbb{R}^n } }
	,
	\vectau ( s_2 ) \wedge \frac { \Delta \vecf } { \| \Delta \vecf \|_{ \mathbb{R}^n } }
	\right\rangle_{ {\bigwedge}^2 \mathbb{R}^n }
\end{align*}
and therefore
\begin{align*}
	&
	\frac \partial { \partial s_1 } \| \vecf ( s_1 ) - \vecf ( s_2 ) \|_{ \mathbb{R}^n }
	\frac \partial { \partial s_2 } \| \vecf ( s_1 ) - \vecf ( s_2 ) \|_{ \mathbb{R}^n }
	\\
	& \quad
	=
	-
	\vectau ( s_1 ) \cdot \vectau ( s_2 )
	+
	\vectau ( s_1 ) \cdot \vectau ( s_2 )
	-
	\left(
	\vectau ( s_1 ) \cdot \frac { \Delta \vecf } { \| \Delta \vecf \|_{ \mathbb{R}^n } }
	\right)
	\left(
	\vectau ( s_2 ) \cdot \frac { \Delta \vecf } { \| \Delta \vecf \|_{ \mathbb{R}^n } }
	\right)
	\\
	& \quad
	=
	-
	\vectau ( s_1 ) \cdot \vectau ( s_2 )
	+
	\left\langle
	\vectau ( s_1 ) \wedge \frac { \Delta \vecf } { \| \Delta \vecf \|_{ \mathbb{R}^n } }
	,
	\vectau ( s_2 ) \wedge \frac { \Delta \vecf } { \| \Delta \vecf \|_{ \mathbb{R}^n } }
	\right\rangle_{ {\bigwedge}^2 \mathbb{R}^n }
	. 
\end{align*}
Finally, we have
\begin{align*}
	&
	\frac { \partial^2 } { \partial s_1 \partial s_2 }
	\Psi ( \| \vecf ( s_1 ) - \vecf ( s_2 ) \|_{ \mathbb{R}^n } )
	\\
	& \quad
	=
	-
	\Psi^{ \prime \prime } ( \| \Delta \vecf \|_{ \mathbb{R}^n } )
	( \vectau ( s_1 ) \cdot \vectau ( s_2 ) )
	\\
	& \quad \qquad
	+ \,
	\left( 
	\Psi^{ \prime \prime } ( \| \Delta \vecf \|_{ \mathbb{R}^n } )
	-
	\frac { \Psi^\prime ( \| \Delta \vecf \|_{ \mathbb{R}^n } ) }
	{ \| \Delta \vecf \|_{ \mathbb{R}^n } }
	\right)
	\left\langle
	\vectau ( s_1 ) \wedge \frac { \Delta \vecf } { \| \Delta \vecf \|_{ \mathbb{R}^n } }
	,
	\vectau ( s_2 ) \wedge \frac { \Delta \vecf } { \| \Delta \vecf \|_{ \mathbb{R}^n } }
	\right\rangle_{ {\bigwedge}^2 \mathbb{R}^n }
	\\
	& \quad
	=
	-
	\frac
	{ \vectau ( s_1 ) \cdot \vectau ( s_2 ) }
	{ \Phi ( \| \Delta \vecf \|_{ \mathbb{R}^n } ) }
	\\
	& \quad \qquad
	+ \,
	\left(
	\frac 1 { \Phi ( \| \Delta \vecf \|_{ \mathbb{R}^n } ) }
	-
	\Lambda ( \| \Delta \vecf \|_{ \mathbb{R}^n } )
	\right)
	\left\langle
	\vectau ( s_1 ) \wedge \frac { \Delta \vecf } { \| \Delta \vecf \|_{ \mathbb{R}^n } }
	,
	\vectau ( s_2 ) \wedge \frac { \Delta \vecf } { \| \Delta \vecf \|_{ \mathbb{R}^n } }
	\right\rangle_{ {\bigwedge}^2 \mathbb{R}^n }
	\\
	& \quad
	=
	-
	\frac
	{ \vectau ( s_1 ) \cdot \vectau ( s_2 ) }
	{ \Phi ( \| \Delta \vecf \|_{ \mathbb{R}^n } ) }
	+
	\mathscr{M}_{ \Phi , 2 } ( \vecf )
\end{align*}
and we conclude that 
\begin{align*}
	&
	\frac 1 { \Phi ( \| \vecf ( s_1 ) - \vecf ( s_2 ) \|_{ \mathbb{R}^n } ) }
	+
	\frac { \partial^2 } { \partial s_1 \partial s_2 } \Psi ( \| \vecf ( s_1 ) - \vecf ( s_2 ) \|_{ \mathbb{R}^n } )
	\\
	& \quad
	=
	\frac { 1 - \vectau ( s_1 ) \cdot \vectau ( s_2 ) } { \Phi ( \| \vecf ( s_1 ) - \vecf ( s_2 ) \|_{ \mathbb{R}^n } ) }
	+
	\mathscr{M}_{ \Phi , 2 } ( \vecf )
	\\
	& \quad
	=
	\mathscr{M}_{ \Phi , 1 } ( \vecf )
	+
	\mathscr{M}_{ \Phi , 2 } ( \vecf ). 
\end{align*}
\qed
\end{proof}
\par
\begin{prop}
Under the assumptions {\rm (A.1)--(A.4)}, it holds that 
\begin{align*}
	&
	\lim_{ \epsilon \to + 0 }
	\int_{ \mathbb{R} / \mathcal{L} \mathbb{Z} }
	\left(
	\int_{ s_1 - \frac { \mathcal{L} } 2 + \epsilon }^{ s_1 - \epsilon }
	+
	\int_{ s_1 + \epsilon }^{ s_1 + \frac { \mathcal{L} } 2 - \epsilon }
	\right)
	\frac { \partial^2 } { \partial s_1 \partial s_2 }
	\left(
	\Psi ( \mathscr{D} ( \vecf ( s_1 ) , \vecf ( s_2 ) )
	-
	\Psi ( \| \vecf ( s_1 ) - \vecf ( s_2 ) \|_{ \mathbb{R}^n } )
	\right)
	d s_2 d s_1
	\\
	& \quad
	=
	2 \mathcal{L} \int_{ \frac { \mathcal{L} } 2 }^\infty
	\frac { dx } { \Phi (x) }. 
\end{align*}
\end{prop}
\begin{proof}
If $ 0 < | s_1 - s_2 | < \frac { \mathcal{L} } 2 $, an easy calculation 
\begin{align*}
	&
	\frac \partial { \partial s_1 }
	\left(
	\Psi ( \mathscr{D} ( \vecf ( s_1 ) , \vecf ( s_2 ) )
	-
	\Psi ( \| \vecf ( s_1 ) - \vecf ( s_2 ) \|_{ \mathbb{R}^n } )
	\right)
	\\
	& \quad
	=
	\Psi^\prime ( | \Delta s | ) \frac { \partial | \Delta s | } { \partial s_1 }
	-
	\Psi^\prime ( \| \Delta \vecf \|_{ \mathbb{R}^n } ) \frac { \partial \| \Delta \vecf \|_{ \mathbb{R}^n } } { \partial s_1 }
	\\
	& \quad
	=
	\Psi^\prime ( | \Delta s | ) \frac { \Delta s } { | \Delta s | }
	-
	\Psi^\prime ( \| \Delta \vecf \|_{ \mathbb{R}^n } )
	\left( \vectau ( s_1 ) \cdot \frac { \Delta \vecf } { \| \Delta \vecf \|_{ \mathbb{R}^n } } \right)
	\\
	& \quad
	=
	\Delta s
	\left\{
	- \Lambda ( | \Delta s | )
	+ \Lambda ( \| \Delta \vecf \|_{ \mathbb{R}^n } )
	\left( \vectau ( s_1 ) \cdot \frac { \Delta \vecf } { \Delta s } \right)
	\right\}
\end{align*}
leads to 
\begin{small}
\begin{align*}
	&
	\left(
	\int_{ s_1 - \frac { \mathcal{L} } 2 + \epsilon }^{ s_1 - \epsilon }
	+
	\int_{ s_1 + \epsilon }^{ s_1 + \frac { \mathcal{L} } 2 - \epsilon }
	\right)
	\frac { \partial^2 } { \partial s_1 \partial s_2 }
	\left(
	\Psi ( \mathscr{D} ( \vecf ( s_1 ) , \vecf ( s_2 ) )
	-
	\Psi ( \| \vecf ( s_1 ) - \vecf ( s_2 ) \|_{ \mathbb{R}^n } )
	\right)
	d s_2
	\\
	& \quad
	=
	\left[
	\Delta s 
	\left\{
	- \Lambda ( | \Delta s | )
	+ \Lambda ( \| \Delta \vecf \|_{ \mathbb{R}^n } )
	\left( \vectau ( s_1 ) \cdot \frac { \Delta \vecf } { \Delta s } \right)
	\right\}
	\right]_{ s_2 = s_1 - \frac { \mathcal{L} } 2 + \epsilon }^{ s_2 = s_1 - \epsilon }
	\\
	& \quad \qquad
	+ \,
	\left[
	\Delta s 
	\left\{
	- \Lambda ( | \Delta s | )
	+ \Lambda ( \| \Delta \vecf \|_{ \mathbb{R}^n } )
	\left( \vectau ( s_1 ) \cdot \frac { \Delta \vecf } { \Delta s } \right)
	\right\}
	\right]_{ s_2 = s_1 + \epsilon }^{ s_2 = s_1 + \frac { \mathcal{L} } 2 - \epsilon }
	\\
	& \quad
	=
	\epsilon
	\left\{
	- \Lambda ( \epsilon )
	+
	\Lambda ( \| \vecf ( s_1 ) - \vecf ( s_1 - \epsilon ) \|_{ \mathbb{R}^n } )
	\left( \vectau ( s_1 ) \cdot \frac { \vecf ( s_1 ) - \vecf ( s_1 - \epsilon ) } \epsilon \right)
	\right\}
	\\
	& \quad \qquad
	- \,
	\left( \frac { \mathcal{L} } 2 - \epsilon \right)
	\left\{
	- \Lambda \left( \frac { \mathcal{L} } 2 - \epsilon \right)
	+
	\Lambda \left( \left\| \vecf ( s_1 ) - \vecf \left( s_1 - \frac { \mathcal{L} } 2 + \epsilon \right) \right\|_{ \mathbb{R}^n } \right)
	\left(
	\vectau ( s_1 ) \cdot
	\frac { \vecf ( s_1 ) - \vecf \left( s_1 - \frac { \mathcal{L} } 2 + \epsilon \right) }
	{ \frac { \mathcal{L} } 2 - \epsilon }
	\right)
	\right\}
	\\
	& \quad \qquad
	- \,
	\left( \frac { \mathcal{L} } 2 - \epsilon \right)
	\left\{
	- \Lambda \left( \frac { \mathcal{L} } 2 - \epsilon \right)
	+
	\Lambda \left( \left\| \vecf ( s_1 ) - \vecf \left( s_1 + \frac { \mathcal{L} } 2 - \epsilon \right) \right\|_{ \mathbb{R}^n } \right)
	\left(
	\vectau ( s_1 ) \cdot
	\frac { \vecf ( s_1 ) - \vecf \left( s_1 + \frac { \mathcal{L} } 2 - \epsilon \right) }
	{ - \frac { \mathcal{L} } 2 + \epsilon }
	\right)
	\right\}
	\\
	& \quad \qquad
	- \,
	( - \epsilon )
	\left\{
	- \Lambda ( \epsilon )
	+
	\Lambda ( \| \vecf ( s_1 ) - \vecf ( s_1 + \epsilon ) \|_{ \mathbb{R}^n } )
	\left( \vectau ( s_1 ) \cdot \frac { \vecf ( s_1 ) - \vecf ( s_1 + \epsilon ) } { - \epsilon } \right)
	\right\}
	\\
	& \quad
	=
	-
	2 \epsilon \Lambda ( \epsilon )
	+
	\Lambda ( \| \vecf ( s_1 ) - \vecf ( s_1 - \epsilon ) \|_{ \mathbb{R}^n } )
	\left\{
	( \vectau ( s_1 ) - \vectau ( s_1 - \epsilon ) ) \cdot ( \vecf ( s_1 ) - \vecf ( s_1 - \epsilon ) )
	\right\}
	\\
	& \quad \qquad
	+ \,
	\Lambda ( \| \vecf ( s_1 ) - \vecf ( s_1 - \epsilon ) \|_{ \mathbb{R}^n } )
	\left\{
	\vectau ( s_1 - \epsilon ) \cdot ( \vecf ( s_1 ) - \vecf ( s_1 - \epsilon ) )
	\right\}
	\\
	& \quad \qquad
	- \,
	\Lambda ( \| \vecf ( s_1 ) - \vecf ( s_1 + \epsilon ) \|_{ \mathbb{R}^n } )
	\left\{
	\vectau ( s_1 ) \cdot ( \vecf ( s_1 ) - \vecf ( s_1 + \epsilon ) )
	\right\}
	\\
	& \quad \qquad
	- \,
	2 \left( \frac { \mathcal{L} } 2 - \epsilon \right)
	\Lambda \left( \frac { \mathcal{L} } 2 - \epsilon \right)
	\\
	& \quad \qquad
	- \,
	\Lambda \left( \left\| \vecf ( s_1 ) - \vecf \left( s_1 - \frac { \mathcal{L} } 2 + \epsilon \right) \right\|_{ \mathbb{R}^n } \right)
	\vectau ( s_1 ) \cdot \left( \vecf ( s_1 ) - \vecf \left( s_1 - \frac { \mathcal{L} } 2 + \epsilon \right) \right)
	\\
	& \quad \qquad
	+ \,
	\Lambda \left( \left\| \vecf ( s_1 ) - \vecf \left( s_1 + \frac { \mathcal{L} } 2 - \epsilon \right) \right\|_{ \mathbb{R}^n } \right)
	\vectau ( s_1 ) \cdot \left( \vecf ( s_1 ) - \vecf \left( s_1 + \frac { \mathcal{L} } 2 - \epsilon \right) \right). 
\end{align*}
\end{small}
We define $ \tilde \Lambda $ by
\[
	\Lambda (x) = \tilde \Lambda ( \sqrt x ). 
\]
From
\begin{align*}
	&
	\Lambda ( \| \vecf ( s_1 ) - \vecf ( s_1 - \epsilon ) \|_{ \mathbb{R}^n } )
	\left\{
	( \vectau ( s_1 ) - \vectau ( s_1 - \epsilon ) ) \cdot ( \vecf ( s_1 ) - \vecf ( s_1 - \epsilon ) )
	\right\}
	\\
	& \quad
	=
	\frac 12 \frac d { d s_1 }
	\int^{ \| \smallvecf ( s_1 ) - \smallvecf ( s_1 - \epsilon ) \|_{ \mathbb{R}^n }^2 } \tilde \Lambda (x) \, dx, 
\end{align*}
we know that
\[
	\int_{ \mathbb{R} / \mathcal{L} \mathbb{Z} }
	\Lambda ( \| \vecf ( s_1 ) - \vecf ( s_1 - \epsilon ) \|_{ \mathbb{R}^n }^2 )
	\left\{
	( \vectau ( s_1 ) - \vectau ( s_1 - \epsilon ) ) \cdot ( \vecf ( s_1 ) - \vecf ( s_1 - \epsilon ) )
	\right\}
	d s_1
	= 0. 
\]
Since $ \vecf $ is periodic, we obtain 
\begin{align*}
	&
	\int_{ \mathbb{R} / \mathcal{L} \mathbb{Z} }
	\left[
	\Lambda ( \| \vecf ( s_1 ) - \vecf ( s_1 - \epsilon ) \|_{ \mathbb{R}^n } )
	\left\{
	\vectau ( s_1 - \epsilon ) \cdot ( \vecf ( s_1 ) - \vecf ( s_1 - \epsilon ) )
	\right\}
	\right.
	\\
	& \quad \qquad
	\left.
	- \,
	\Lambda ( \| \vecf ( s_1 ) - \vecf ( s_1 + \epsilon ) \|_{ \mathbb{R}^n } )
	\left\{
	\vectau ( s_1 ) \cdot ( \vecf ( s_1 ) - \vecf ( s_1 + \epsilon ) )
	\right\}
	\right]
	d s_1
	\\
	& \quad
	=
	\int_{ \mathbb{R} / \mathcal{L} \mathbb{Z} }
	\left[
	\Lambda ( \| \vecf ( s_1 + \epsilon ) - \vecf ( s_1 ) \|_{ \mathbb{R}^n } )
	\left\{
	\vectau ( s_1 ) \cdot ( \vecf ( s_1 + \epsilon ) - \vecf ( s_1 ) )
	\right\}
	\right.
	\\
	& \quad \qquad
	\left.
	- \,
	\Lambda ( \| \vecf ( s_1 ) - \vecf ( s_1 + \epsilon ) \|_{ \mathbb{R}^n } )
	\left\{
	\vectau ( s_1 ) \cdot ( \vecf ( s_1 ) - \vecf ( s_1 + \epsilon ) )
	\right\}
	\right]
	d s_1
	\\
	& \quad
	=
	2
	\int_{ \mathbb{R} / \mathcal{L} \mathbb{Z} }
	\Lambda ( \| \vecf ( s_1 + \epsilon ) - \vecf ( s_1 ) \|_{ \mathbb{R}^n } )
	\left\{
	\vectau ( s_1 ) \cdot ( \vecf ( s_1 + \epsilon ) - \vecf ( s_1 ) )
	\right\}
	d s_1
	\\
	& \quad
	=
	2
	\int_{ \mathbb{R} / \mathcal{L} \mathbb{Z} }
	\Lambda ( \| \vecf ( s_1 + \epsilon ) - \vecf ( s_1 ) \|_{ \mathbb{R}^n } )
	\left(
	\vectau ( s_1 ) \cdot \int_{ s_1 }^{ s_1 + \epsilon } \vectau ( s_2 )
	\right)
	d s_2
	d s_1
\end{align*}
and therefore 
\begin{align*}
	&
	\int_{ \mathbb{R} / \mathcal{L} \mathbb{Z} }
	\left[
	- 2 \epsilon \Lambda ( \epsilon )
	+
	\Lambda ( \| \vecf ( s_1 ) - \vecf ( s_1 - \epsilon ) \|_{ \mathbb{R}^n } )
	\left\{
	\vectau ( s_1 - \epsilon ) \cdot ( \vecf ( s_1 ) - \vecf ( s_1 - \epsilon ) )
	\right\}
	\right.
	\\
	& \quad \qquad
	\left.
	- \,
	\Lambda ( \| \vecf ( s_1 ) - \vecf ( s_1 + \epsilon ) \|_{ \mathbb{R}^n } )
	\left\{
	\vectau ( s_1 ) \cdot ( \vecf ( s_1 ) - \vecf ( s_1 + \epsilon ) )
	\right\}
	\right]
	d s_1
	\\
	& \quad
	=
	2 \epsilon
	\int_{ \mathbb{R} / \mathcal{L} \mathbb{Z} }
	\left(
	\Lambda ( \| \vecf ( s_1 ) - \vecf ( s_1 + \epsilon ) \|_{ \mathbb{R}^n } )
	-
	\Lambda ( \epsilon )
	\right) d s_1
	\\
	& \quad \qquad
	- \,
	\int_{ \mathbb{R} / \mathcal{L} \mathbb{Z} }
	\Lambda ( \| \vecf ( s_1 + \epsilon ) - \vecf ( s_1 ) \|_{ \mathbb{R}^n } )
	\int_{ s_1 }^{ s_1 + \epsilon }
	\| \vectau ( s_1 ) - \vectau ( s_2 ) \|_{ \mathbb{R}^n }^2
	d s_2 d s_1
\end{align*}
holds. If $ \epsilon \to + 0 $, the above amount converges to $ 0 $ from (A.4) and obviously 
\[
	-
	\int_{ \mathbb{R} / \mathcal{L} \mathbb{Z} }
	2
	\left( \frac { \mathcal{L} } 2 - \epsilon \right)
	\Lambda \left( \frac { \mathcal{L} } 2 - \epsilon \right)
	d s_1
	\to
	- \mathcal{L}^2 \Lambda \left( \frac { \mathcal{L} } 2 \right)
	=
	2 \mathcal{L} \int_{ \frac { \mathcal{L} } 2 }^\infty
	\frac { dx } { \Phi (x) }. 
\]
follows.
From the Lebesgue's convergence theorem, it holds that 
\begin{align*}
	&
	\int_{ \mathbb{R} / \mathcal{L} \mathbb{Z} }
	\left\{
	-
	\Lambda \left( \left\| \vecf ( s_1 ) - \vecf \left( s_1 - \frac { \mathcal{L} } 2 + \epsilon \right) \right\|_{ \mathbb{R}^n } \right)
	\vectau ( s_1 ) \cdot \left( \vecf ( s_1 ) - \vecf \left( s_1 - \frac { \mathcal{L} } 2 + \epsilon \right) \right)
	\right.
	\\
	& \quad \qquad
	\left.
	+ \,
	\Lambda \left( \left\| \vecf ( s_1 ) - \vecf \left( s_1 + \frac { \mathcal{L} } 2 - \epsilon \right) \right\|_{ \mathbb{R}^n } \right)
	\vectau ( s_1 ) \cdot \left( \vecf ( s_1 ) - \vecf \left( s_1 + \frac { \mathcal{L} } 2 - \epsilon \right) \right)
	\right\}
	d s_1
	\to 0
\end{align*}
as $ \epsilon \to + 0 $ remarking that $ \vecf \left( s_1 + \frac { \mathcal{L} } 2 \right)
= \vecf \left( s_1 - \frac { \mathcal{L} } 2 \right) $. 
\qed
\end{proof}
\par
From this lemma, a principal value integral
\[
	\lim_{ \epsilon \to + 0 }
	\int_{ \mathbb{R} / \mathcal{L} \mathbb{Z} }
	\left(
	\int_{ s_1 - \frac { \mathcal{L} } 2 + \epsilon }^{ s_1 - \epsilon }
	+
	\int_{ s_1 + \epsilon }^{ s_1 + \frac { \mathcal{L} } 2 - \epsilon }
	\right)
	\left(
	\mathscr{M}_{ \Phi , 1 } ( \vecf )
	+
	\mathscr{M}_{ \Phi , 2 } ( \vecf )
	\right)
	d s_2 d s_1
\]
converges under the conditions (A.1)--(A.4) and 
\begin{align*}
	\mathcal{E}_{ \Phi } ( \vecf )
	= & \
	\lim_{ \epsilon \to + 0 }
	\int_{ \mathbb{R} / \mathcal{L} \mathbb{Z} }
	\left(
	\int_{ s_1 - \frac { \mathcal{L} } 2 + \epsilon }^{ s_1 - \epsilon }
	+
	\int_{ s_1 + \epsilon }^{ s_1 + \frac { \mathcal{L} } 2 - \epsilon }
	\right)
	\left(
	\mathscr{M}_{ \Phi , 1 } ( \vecf )
	+
	\mathscr{M}_{ \Phi , 2 } ( \vecf )
	\right)
	d s_2 d s_1
	\\
	& \quad
	+ \,
	2 \mathcal{L} \int_{ \frac { \mathcal{L} } 2 }^\infty \frac { dx } { \Phi (x) }
\end{align*}
holds. We can consider that this is a preliminary step for a decomposed theorem, that is, this is equivalent to the cosine formula in case of the M\"{o}bius energy. 
From now on, we show that 
$ \mathscr{M}_{ \Phi , 1 } ( \vecf ) $ and $ \mathscr{M}_{ \Phi , 2 } ( \vecf ) $ belong to $ L^1 ( ( \mathbb{R} / \mathcal{L} \mathbb{Z} )^2 ) $ when we assume the condition (A.5). 
\par
Firstly we show absolute integrability of 
$ \mathscr{M}_{ \Phi , 1 } ( \vecf ) $. 
There exists a constant $ \lambda \in ( 0,1 ) $ with $ \lambda \mathrm{dist}_{ \mathbb{R} / \mathcal{L} \mathbb{Z} } ( s_1 , s_2 ) = \mathscr{D} ( \vecf ( s_1 ) , \vecf ( s_2 ) ) \leqq \| \Delta \vecf \|_{ \mathbb{R}^n } $ since  $ \vecf $ is bi-Lipschitz and we obtain 
\[
	0 \leqq
	\mathscr{M}_{ \Phi , 1 } ( \vecf )
	=
	\frac { \| \Delta \vectau \|_{ \mathbb{R}^n }^2 } { 2 \Phi ( \| \Delta \vecf \|_{ \mathbb{R}^n } ) }
	\leqq
	\frac { \| \Delta \vectau \|_{ \mathbb{R}^n }^2 } { 2 \Phi ( \lambda \mathrm{dist}_{ \mathbb{R} / \mathcal{L} \mathbb{Z} } ( s_1 , s_2 ) ) }
	\leqq
	\frac { \| \Delta \vectau \|_{ \mathbb{R}^n }^2 }
	{ 2 C( \lambda , \mathcal{L} ) \Phi ( \mathrm{dist}_{ \mathbb{R} / \mathcal{L} \mathbb{Z} } ( s_1 , s_2 ) ) ) }. 
\]
The condition (A.3) leads to $ \vecf \in W_\Phi $ and therefore we know that
$ \mathscr{M}_{ \Phi , 1 } ( \vecf ) \in L^1 ( ( \mathbb{R} / \mathcal{L} \mathbb{Z} )^2 ) $. 
The following elementary identity shows absolute integrability of $ \mathscr{M}_{ \Phi , 2 } ( \vecf ) $. 
\begin{lem}
For $ \veca $,
$ \vecb $,
$ \vecc \in \mathbb{R}^n $, it holds that
\[
	\left| ( \veca - \vecb ) \cdot \vecc \right|^2
	+
	\left\| ( \veca + \vecb ) \wedge \vecc \right\|_{ \mathbb{R}^n}^2
	=
	\| \veca - \vecb \|_{ \mathbb{R}^n }^2
	\| \vecc \|_{ \mathbb{R}^n }^2
	+
	4 \langle \veca \wedge \vecc , \vecb \wedge \vecc \rangle_{ {\bigwedge}^2 \mathbb{R}^n }. 
\]
\end{lem}
\begin{proof}
Calculating straightforward, we obtain the conclusion
\begin{align*}
	&
	\left| ( \veca - \vecb ) \cdot \vecc \right|^2
	+
	\left\| ( \veca + \vecb ) \wedge \vecc \right\|_{ \mathbb{R}^n}^2
	\\
	& \quad
	=
	\left| ( \veca - \vecb ) \cdot \vecc \right|^2
	+
	\det \left(
	\begin{array}{cc}
	\| \veca + \vecb \|_{ \mathbb{R}^n }^2 &
	( \veca + \vecb ) \cdot \vecc \\
	( \veca + \vecb ) \cdot \vecc &
	\| \vecc \|_{ \mathbb{R}^n }
	\end{array}
	\right)
	\\
	& \quad
	=
	\left| ( \veca - \vecb ) \cdot \vecc \right|^2
	+
	\| \veca + \vecb \|_{ \mathbb{R}^n }^2
	\| \vecc \|_{ \mathbb{R}^n }^2
	-
	\left| ( \veca + \vecb ) \cdot \vecc \right|^2
	\\
	& \quad
	=
	\| \veca + \vecb \|_{ \mathbb{R}^n }^2
	\| \vecc \|_{ \mathbb{R}^n }^2
	-
	4
	( \veca \cdot \vecc )
	( \vecb \cdot \vecc )
	\\
	& \quad
	=
	\| \veca - \vecb \|_{ \mathbb{R}^n }^2
	\| \vecc \|_{ \mathbb{R}^n }^2
	+
	4 ( \veca \cdot \vecb )
	\| \vecc \|_{ \mathbb{R}^n }^2
	-
	4
	( \veca \cdot \vecc )
	( \vecb \cdot \vecc )
	\\
	& \quad
	=
	\| \veca - \vecb \|_{ \mathbb{R}^n }^2
	\| \vecc \|_{ \mathbb{R}^n }^2
	+
	4 \det
	\left(
	\begin{array}{cc}
	\veca \cdot \vecb &  \veca \cdot \vecc \\
	\vecc \cdot \vecb & \| \vecc \|_{ \mathbb{R}^n }^2
	\end{array}
	\right)
	\\
	& \quad
	=
	\| \veca - \vecb \|_{ \mathbb{R}^n }^2
	\| \vecc \|_{ \mathbb{R}^n }^2
	+
	4 \langle \veca \wedge \vecc , \vecb \wedge \vecc \rangle_{ {\bigwedge}^2 \mathbb{R}^n }. 
\end{align*}
\qed
\end{proof}
\par
Put
\[
	\veca = \vectau ( s_1 ) ,
	\quad
	\vecb = \vectau ( s_2 ) ,
	\quad
	\vecc = \frac { \Delta \vecf } { \| \Delta \vecf \|_{ \mathbb{R}^n } }. 
\]
An easy calculation 
\begin{align*}
	&
	\left|
	\Delta \vectau \cdot \frac { \Delta \vecf } { \| \Delta \vecf \|_{ \mathbb{R}^n } }
	\right|^2
	+
	\left\|
	( \vectau ( s_1 ) + \vectau ( s_2 ) ) \wedge
	\frac { \Delta \vecf } { \| \Delta \vecf \|_{ \mathbb{R}^n } }
	\right\|_{ \mathbb{R}^n }^2
	\\
	& \quad
	=
	\| \Delta \vectau \|_{ \mathbb{R}^n }^2
	+
	4 \left\langle
	\vectau ( s_1 ) \wedge
	\frac { \Delta \vecf } { \| \Delta \vecf \|_{ \mathbb{R}^n } }
	,
	\vectau ( s_2 ) \wedge
	\frac { \Delta \vecf } { \| \Delta \vecf \|_{ \mathbb{R}^n } }
	\right\rangle_{ { \bigwedge }^2 \mathbb{R}^n }
\end{align*}
leads to 
\begin{align}
	\label{M1M2}
	&
	\mathscr{M}_{ \Phi , 1 } ( \vecf )
	+
	\mathscr{M}_{ \Phi , 2 } ( \vecf )
	\\
	\nonumber
	& \quad
	=
	\frac 14 \left( \frac 1 { \Phi (x) } - \Lambda (x) \right)
	\left\{
	\left|
	\Delta \vectau \cdot \frac { \Delta \vecf } { \| \Delta \vecf \|_{ \mathbb{R}^n } }
	\right|^2
	+
	\left\|
	( \vectau ( s_1 ) + \vectau ( s_2 ) ) \wedge
	\frac { \Delta \vecf } { \| \Delta \vecf \|_{ \mathbb{R}^n } }
	\right\|_{ \mathbb{R}^n }^2
	\right\}
	\\
	\nonumber
	& \quad \qquad
	+ \,
	\frac 14 \left( \frac 1 { \Phi (x) } + \Lambda (x) \right)
	\| \Delta \vectau \|_{ \mathbb{R}^n }^2. 
\end{align}
From the fact that $ \Phi (x) > 0 $ and $ \Lambda (x) < 0 $, the right hand side is a non-negative function if (b) of (A.5) holds. 
Accordingly $ \mathscr{M}_{ \Phi , 1 } ( \vecf ) + \mathscr{M}_{ \Phi , 2 } ( \vecf ) $ is absolutely integrable from the integrability in the sense of principal value. 
Hence, so is $ \mathscr{M}_{ \Phi , 2 } ( \vecf ) = ( \mathscr{M}_{ \Phi , 1 } ( \vecf ) + \mathscr{M}_{ \Phi , 2 } ( \vecf ) ) - \mathscr{M}_{ \Phi , 1 } ( \vecf ) $ and we prove theorem \ref{Main Theorem}. 
\qed
\begin{eg}
When
$ \Phi ( x ) = x^\alpha $ {\rm (}$ \alpha \in [ 2 , 3 ) ${\rm )}, 
it follows that
\begin{align*}
	\mathcal{E}_{ x^\alpha } ( \vecf )
	= & \
	\mathcal{E}_{ x^\alpha , 1 } ( \vecf )
	+
	\mathcal{E}_{ x^\alpha , 2 } ( \vecf )
	+
	\frac { 2^\alpha } { ( \alpha - 1 ) \mathcal{L}^{ \alpha - 2 } }
	,
	\\
	\mathcal{E}_{ x^\alpha , 1 } ( \vecf )
	= & \
	\iint_{ ( \mathbb{R} / \mathcal{L} \mathbb{Z} )^2 }
	\frac { \| \Delta \vectau \|_{ \mathbb{R}^n }^2 } { 2 \| \Delta \vecf \|_{ \mathbb{R}^n }^\alpha }
	\,
	d s_1 d s_2
	,
	\\
	\mathcal{E}_{ x^\alpha , 2 } ( \vecf )
	= & \
	\iint_{ ( \mathbb{R} / \mathcal{L} \mathbb{Z} )^2 }
	\frac \alpha { ( \alpha - 1 ) \| \Delta \vecf \|_{ \mathbb{R}^n }^\alpha }
	\left\langle
	\vectau ( s_1 ) \wedge \frac { \Delta \vecf } { \| \Delta \vecf \|_{ \mathbb{R}^n } }
	,
	\vectau ( s_2 ) \wedge \frac { \Delta \vecf } { \| \Delta \vecf \|_{ \mathbb{R}^n } }
	\right\rangle_{ {\bigwedge}^2 \mathbb{R}^n }
	d s_1 d s_2. 
\end{align*}
In particular, the case of  $ \alpha = 2 $ corresponds to the decomposition described in {\rm \S~1}. 
\end{eg}
\begin{rem}
When $ \Phi ( x ) = x^\alpha $ {\rm (}$ \alpha \in [ 2 , 3 ) ${\rm )}, a constant $ \displaystyle{ \frac { 2^\alpha } { ( \alpha - 1 ) \mathcal{L}^{ \alpha - 2 } } } $ is not energy value of right circles except the case of $ \alpha = 2 $. 
However, we can calculate energy value of right circles using decomposition theorem. 
Let $ \vecf $ be a right circle. It follows that
\[
	\left|
	\Delta \vectau \cdot \frac { \Delta \vecf } { \| \Delta \vecf \|_{ \mathbb{R}^n } }
	\right|^2
	+
	\left\|
	( \vectau ( s_1 ) + \vectau ( s_2 ) ) \wedge
	\frac { \Delta \vecf } { \| \Delta \vecf \|_{ \mathbb{R}^n } }
	\right\|_{ \mathbb{R}^n }^2
	\equiv 0. 
\]
If $ \Phi (x) = x^\alpha $, 
a relation
$ \displaystyle{ \frac 1 { \Phi (x) } + \Lambda (x) = \frac { \alpha - 2 } { ( \alpha - 1 ) x^\alpha } } $ leads to 
\begin{align*}
	\mathcal{E}_{ x^\alpha } ( \vecf )
	= & \
	\iint_{ ( \mathbb{R} / \mathcal{L} \mathbb{Z} )^2 }
	\left(
	\mathscr{M}_{ x^\alpha , 1 } ( \vecf )
	+
	\mathscr{M}_{ x^\alpha , 2 } ( \vecf )
	\right)
	d s_1 d s_2
	+
	\frac 2 { ( \alpha - 1 ) \mathcal{L}^{ \alpha - 2 } }
	\\
	= & \
	\frac { \alpha - 2 } { 2 ( \alpha - 1 ) }
	\mathcal{E}_{ x^\alpha , 1 } ( \vecf )
	+
	\frac 2 { ( \alpha - 1 ) \mathcal{L}^{ \alpha - 2 } }
\end{align*}
from \pref{M1M2}. 
It holds that 
\[
	\| \Delta \vecf \|_{ \mathbb{R}^n }^2
	=
	\frac { \mathcal{L}^2 } { \pi^2 }
	\sin^2 \frac \pi { \mathcal{L} } \Delta s
	,
	\quad
	\| \Delta \vectau \|_{ \mathbb{R}^n }^2
	=
	4
	\sin^2 \frac \pi { \mathcal{L} } \Delta s
\]
for a right circle with total length $ \mathcal{L} $. 
We obtain 
\begin{align*}
	\mathcal{E}_{ x^\alpha , 1 } ( \vecf )
	= & \
	\frac { 2 \pi^\alpha } { \mathcal{L}^\alpha }
	\int_0^{ \mathcal{L} }
	\int_{ - \frac { \mathcal{L} } 2 }^{ \frac { \mathcal{L} } 2 }
	\left\{ \sin^2 \frac \pi { \mathcal{L} } ( s + h ) \right\}^{ 1 - \frac \alpha 2 }
	d h d s
	\\
	= & \
	\frac { 4 \pi^\alpha } { \mathcal{L}^{ \alpha - 1 }}
	\int_0^{ \frac { \mathcal{L} } 2 }
	\sin^{ 2 - \alpha } \frac \pi { \mathcal{L} } s
	\, d s
	\\
	= & \
	\frac { 4 \pi^{ \alpha - 1 } } { \mathcal{L}^{ \alpha - 2 } }
	\int_0^{ \frac \pi 2 }
	\sin^{ 2 - \alpha } \theta \, d \theta
	\\
	= & \
	\frac { 2 \pi^{ \alpha - \frac 12 } } { \mathcal{L}^{ \alpha - 2 } }
	\frac
	{ {\mit \Gamma} \left( \frac { 3 - \alpha } 2 \right) }
	{ {\mit \Gamma} \left( \frac { 4 - \alpha } 2 \right) }
\end{align*}
using
\[
	\int_0^{ \frac \pi 2 }
	\sin^{ 2x - 1 } \theta \cos^{ 2y - 1 } \theta \, d \theta
	=
	\frac 12 \frac { {\mit \Gamma} ( x ) {\mit \Gamma} (y) }
	{ {\mit \Gamma} ( x + y ) }. 
\]
Consequently, we have
\[
	\mathcal{E}_{ x^\alpha } ( \vecf )
	=
	\frac 1 { ( \alpha - 1 ) \mathcal{L}^{ \alpha - 2 } }
	\left\{
	\frac { ( \alpha - 2 ) \pi^{ \alpha - \frac 12 } {\mit \Gamma} \left( \frac { 3 - \alpha } 2 \right) }
	{ {\mit \Gamma} \left( \frac { 4 - \alpha } 2 \right) }
	+
	2^\alpha
	\right\}. 
\]
This value also follows from the result of Brylinski \cite{Brylinski}. 
Note that global minimizers of $ \Phi_{ x^\alpha } $ with fixed total length are right circles,
which is shown by Abrams-Cantarella-Fu-Ghomi-Howard \cite{ACFGH}.
\end{rem}
\section{Variational formulae}
\label{Variational formulae}
\par
In this section, we derive variational formulae of the decomposed energies. 
Setting
\begin{align*}
	Q_{1,i} \vecv
	= & \
	\Delta \vecv^\prime
	= \vecv_1 - \vecv_2 ,
	\\
	Q_{2,i } \vecv
	= & \
	( -1 )^{ i-1 }
	2 \left\{ \vecv_i^\prime - ( R_1 \vecf \cdot \vectau_i ) R_1 \vecv \right\}
	,
	\\
	R_1 \vecv
	= & \
	\frac { | \Delta s | \Delta \vecv } { \| \Delta \vecf \|_{ \mathbb{R}^n } \Delta s }
	,
	\\
	R_2 \vecv
	= & \
	\frac 12 \left( \vecv_1^\prime + \vecv_2^\prime \right)
	,
	\\
	\Phi_1 (x) = & \
	2 \Phi (x) ,
	\\
	\Phi_2 (x) = & \
	- 4 \left( \frac 1 { \Phi (x) } - \Lambda (x) \right)^{-1}, 
\end{align*}
we rewrite $ \mathscr{M}_{ \Phi , j } ( \vecf ) $ using them as
\[
	\mathscr{M}_{ \Phi , j } ( \vecf )
	=
	\frac { Q_{ j,1 } ( \vecf ) \cdot Q_{ j,2 } ( \vecf ) } { \Phi_j ( \| \Delta \vecf \|_{ \mathbb{R}^n } ) }. 
\]
Furthermore we put
\begin{align*}
	S_{1,i} ( \vecv , \vecw )
	= & \
	R_2 \vecv \cdot Q_{1,i} \vecw
	+
	Q_{1,i} \vecv \cdot R_2 \vecw
	,
	\\
	S_{2,i} ( \vecv , \vecw )
	= & \
	R_1 \vecv \cdot Q_{2,i} \vecw
	+
	Q_{2,i} \vecv \cdot R_1 \vecw. 
\end{align*}
We assume that $ \Phi_j $ is differentiable and put
\[
	\Xi_i (x)
	=
	\frac { x \Phi_j^\prime (x) } { \Phi_j (x) }. 
\]
We denote $ \mathscr{G}_{ \Phi , j } $ and $ \mathscr{H}_{ \Phi , j } $ as integrand of the first and second variational formulae of $ \mathscr{M}_{ \Phi , j } $ respectively, that is, 
\begin{align*}
	\mathscr{G}_{ \Phi , j } ( \vecf ) [ \vecphi ] \, d s_1 d s_2
	= & \
	\delta ( \mathscr{M}_{ \Phi , j } ( \vecf ) \, d s_1 d s_2 ) [ \vecphi ]
	\\
	\mathscr{H}_{ \Phi , j } ( \vecf ) [ \vecphi ] \, d s_1 d s_2
	= & \
	\delta^2 ( \mathscr{M}_{ \Phi , j } ( \vecf ) \, d s_1 d s_2 ) [ \vecphi , \vecpsi ]. 
\end{align*}
\begin{thm}
For any $ \mathcal{L} > 0 $, we suppose the conditions 
{\rm (A.1)--(A.5)}. 
If $ \Phi_j \in C^1 ( 0 , \infty ) $ it holds that
\[
	\mathscr{G}_{ \Phi , j } ( \vecf ) [ \vecphi ]
	=
	\frac{
	Q_{ j,1 } \vecf \cdot Q_{ j,2 } \vecphi + Q_{ j,2 } \vecf \cdot Q_{ j,1 } \vecphi
	}
	{ \Phi_j ( \| \Delta \vecf \|_{ \mathbb{R}^n } ) }
	-
	\frac { \mathscr{M}_{ \Phi , j } ( \vecf )
	\Phi_j^\prime ( \| \Delta \vecf \|_{ \mathbb{R}^n } ) \Delta \vecf \cdot \Delta \vecphi }
	{ \| \Delta \vecf \|_{ \mathbb{R}^n } \Phi_j ( \| \Delta \vecf \|_{ \mathbb{R}^n } ) }. 
\]
If $ \Phi_j \in C^2 ( 0 , + \infty ) $, It holds that 
\begin{align*}
	&
	\mathscr{H}_{ \Phi , j } ( \vecf ) [ \vecphi , \vecpsi ]
	\\
	& \quad
	=
	\frac {
	Q_{ j,1 } \vecphi \cdot Q_{j,2} \vecpsi
	+
	Q_{ j,2 } \vecphi \cdot Q_{j,1} \vecpsi
	-
	\left(
	S_{ j,1 } ( \vecf , \vecphi ) S_{ j,2 } ( \vecf , \vecpsi )
	+
	S_{ j,2 } ( \vecf , \vecphi ) S_{ j,1 } ( \vecf , \vecpsi )
	\right)
	}
	{ \Phi_j ( \| \Delta \vecf \|_{ \mathbb{R}^n } ) }
	\\
	& \quad \quad
	- \,
	\mathscr{G}_{ \Phi , j } ( \vecf ) [ \vecphi ]
	\Xi_j ( \| \Delta \vecf \|_{ \mathbb{R}^n } )
	\frac { \Delta \vecf \cdot \Delta \vecpsi } { \| \Delta \vecf \|_{ \mathbb{R}^n }^2 }
	-
	\mathscr{G}_{ \Phi , j } ( \vecf ) [ \vecpsi ]
	\Xi_j ( \| \Delta \vecf \|_{ \mathbb{R}^n } )
	\frac { \Delta \vecf \cdot \Delta \vecphi } { \| \Delta \vecf \|_{ \mathbb{R}^n }^2 }
	\\
	& \quad \quad
	- \,
	\mathscr{M}_{ \Phi , j } ( \vecf )
	\| \Delta \vecf \|_{ \mathbb{R}^n } \Xi_j^\prime ( \| \Delta \vecf \|_{ \mathbb{R}^n } )
	\frac { ( \Delta \vecf \cdot \Delta \vecphi ) ( \Delta \vecf \cdot \Delta \vecpsi ) } { \| \Delta \vecf \|_{ \mathbb{R}^n }^4 }
	\\
	& \quad \quad
	- \,
	\mathscr{M}_{ \Phi , j } ( \vecf ) \Xi_j ( \| \Delta \vecf \|_{ \mathbb{R}^n } )
	\frac { \Delta \vecphi \cdot \Delta\vecpsi } { \| \Delta \vecf \|_{ \mathbb{R}^n }^2 }
	\\
	& \quad \quad
	+ \,
	\mathscr{M}_{ \Phi , j } ( \vecf )
	\left(
	2 \Xi_j ( \| \Delta \vecf \|_{ \mathbb{R}^n } )
	- \Xi_j ( \| \Delta \vecf \|_{ \mathbb{R}^n } )^2
	- \| \Delta \vecf \|_{ \mathbb{R}^n } \Xi_j^\prime ( \| \Delta \vecf \|_{ \mathbb{R}^n } )
	\right)
	\frac { ( \Delta \vecf \cdot \Delta \vecphi ) ( \Delta \vecf \cdot \Delta \vecpsi ) }
	{ \| \Delta \vecf \|_{ \mathbb{R}^n }^4 }. 
\end{align*}
\end{thm}
\begin{proof}
Here we use the same way for proving this theorem as for the M\"{o}bius energy in \cite{IshizekiNagasawaII}.
From
\begin{align*}
	& \quad
	\mathscr{G}_{ \Phi , j } ( \vecf ) [ \vecphi ] \, d s_1 d s_2
	\\
	& \quad
	=
	\delta ( \mathscr{M}_{ \Phi , j } ( \vecf ) \, d s_1 d s_2 ) [ \vecphi ]
	\\
	& \quad
	=
	\delta \left( \frac { Q_{ j,1 } ( \vecf ) \cdot Q_{ j,2 } ( \vecf ) } { \Phi_j ( \| \Delta \vecf \|_{ \mathbb{R}^n } ) } \, d s_1 d s_2 \right) [ \vecphi ]
	\\
	& \quad
	=
	\left\{
	\frac { \delta ( Q_{ j,1 } \vecf \cdot Q_{ j,2 } \vecf ) [ \vecphi ] }
	{ \Phi_j ( \| \Delta \vecf \|_{ \mathbb{R}^n } ) }
	+
	( Q_{ j,1 } \vecf \cdot Q_{ j,2 } \vecf )
	\delta \left( \frac 1 { \Phi_j ( \| \Delta \vecf \|_{ \mathbb{R}^n } ) } \right) [ \vecphi ]
	\right\}
	d s_1 d s_2
	\\
	& \qquad \quad
	+ \,
	\frac { Q_{ j,1 } ( \vecf ) \cdot Q_{ j,2 } ( \vecf ) }
	{ \Phi_j ( \| \Delta \vecf \|_{ \mathbb{R}^n } ) }
	\delta ( d s_1 d s_2 ) [ \vecphi ]
	\\
	& \quad
	=
	\left\{
	\frac{
	Q_{ j,1 } \vecf \cdot Q_{ j,2 } \vecphi + Q_{ j,2 } \vecf \cdot Q_{ j,1 } \vecphi
	}
	{ \Phi_j ( \| \Delta \vecf \|_{ \mathbb{R}^n } ) }
	-
	\frac { ( Q_{ j,1 } \vecf \cdot Q_{ j,2 } \vecf )
	\Phi_j^\prime ( \| \Delta \vecf \|_{ \mathbb{R}^n } ) \Delta \vecf \cdot \Delta \vecphi }
	{ \| \Delta \vecf \|_{ \mathbb{R}^n } \Phi_j ( \| \Delta \vecf \|_{ \mathbb{R}^n } )^2 }
	\right\}
	d s_1 d s_2, 
\end{align*}
it holds that
\[
	\mathscr{G}_{ \Phi , j } ( \vecf ) [ \vecphi ]
	=
	\frac{
	Q_{ j,1 } \vecf \cdot Q_{ j,2 } \vecphi + Q_{ j,2 } \vecf \cdot Q_{ j,1 } \vecphi
	}
	{ \Phi_j ( \| \Delta \vecf \|_{ \mathbb{R}^n } ) }
	-
	\frac { \mathscr{M}_{ \Phi , j } ( \vecf )
	\Phi_j^\prime ( \| \Delta \vecf \|_{ \mathbb{R}^n } ) \Delta \vecf \cdot \Delta \vecphi }
	{ \| \Delta \vecf \|_{ \mathbb{R}^n } \Phi_j ( \| \Delta \vecf \|_{ \mathbb{R}^n } ) }. 
\]
We put 
\[
	\mathscr{P}_j ( \vecf ) [ \vecphi ]
	=
	Q_{ j,1 } \vecf \cdot Q_{ j,2 } \vecphi + Q_{ j,2 } \vecf \cdot Q_{ j,1 } \vecphi, 
\]
then we obtain
\[
	\mathscr{G}_{ \Phi , j } ( \vecf ) [ \vecphi ]
	=
	\frac{
	\mathscr{P}_j ( \vecf ) [ \vecphi ]
	}
	{ \Phi_j ( \| \Delta \vecf \|_{ \mathbb{R}^n } ) }
	-
	\frac { \mathscr{M}_{ \Phi , j } ( \vecf ) \Xi_j ( \| \Delta \vecf \|_{ \mathbb{R}^n } )
	 \Delta \vecf \cdot \Delta \vecphi }
	{ \| \Delta \vecf \|_{ \mathbb{R}^n }^2 }
\]
using $ \mathscr{P}_j $ and $ \Xi_j $. 
An calculation
\begin{align*}
	&
	\delta \left( \mathscr{G}_{ \Phi , j } ( \vecf ) [ \vecphi ] \right) [ \vecpsi ]
	\\
	& \quad
	=
	\frac { \delta \left( \mathscr{P}_j ( \vecf ) [ \vecphi ] \right) [ \vecpsi ] }
	{ \Phi_j ( \| \Delta \vecf \|_{ \mathbb{R}^n } ) }
	+
	\mathscr{P}_j ( \vecf ) [ \vecphi ]
	\delta \left( \frac 1 { \Phi_j ( \| \Delta \vecf \|_{ \mathbb{R}^n } ) } \right) [ \vecpsi ]
	\\
	& \quad \quad
	- \,
	\left[
	\left\{ \delta \left( \mathscr{M}_{ \Phi , j } ( \vecf ) [ \vecphi ] \right) [ \vecpsi ]
	\right\}
	\Xi_j ( \| \Delta \vecf \|_{ \mathbb{R}^n } )
	+
	\mathscr{M}_{ \Phi , j } ( \vecf )
	\left\{ \delta \left( \Xi_j ( \| \Delta \vecf \|_{ \mathbb{R}^n } ) \right) [ \vecpsi ] \right\}
	\right]
	\frac { \Delta \vecf \cdot \Delta \vecphi } { \| \Delta \vecf \|_{ \mathbb{R}^n }^2 }
	\\
	& \quad \quad
	- \,
	\mathscr{M}_{ \Phi , j } ( \vecf ) \Xi_j ( \| \Delta \vecf \|_{ \mathbb{R}^n } )
	\delta \left( \frac { \Delta \vecf \cdot \Delta \vecphi } { \| \Delta \vecf \|_{ \mathbb{R}^n }^2 } \right) [ \vecpsi ]
	\\
	& \quad
	=
	\frac 1
	{ \Phi_j ( \| \Delta \vecf \|_{ \mathbb{R}^n } ) }
	\left\{
	Q_{ j,1 } \vecphi \cdot Q_{j,2} \vecpsi
	+
	Q_{ j,2 } \vecphi \cdot Q_{j,1} \vecpsi
	\right.
	\\
	& \qquad \qquad
	\left.
	- \,
	\mathscr{P}_j ( \vecf ) [ \vecphi ]
	\left( \vectau_1 \cdot \vecpsi_1^\prime + \vectau_2 \cdot \vecpsi_2^\prime \right)
	-
	\left(
	S_{ j,1 } ( \vecf , \vecphi ) S_{ j,2 } ( \vecf , \vecpsi )
	+
	S_{ j,2 } ( \vecf , \vecphi ) S_{ j,1 } ( \vecf , \vecpsi )
	\right)
	\right\}
	\\
	& \quad \quad
	- \,
	\frac {
	\mathscr{P}_j ( \vecf ) [ \vecphi ]
	\Xi_j ( \| \Delta \vecf \|_{ \mathbb{R}^n } ) ( \Delta \vecf \cdot \Delta \vecpsi )
	}
	{ \| \Delta \vecf \|_{ \mathbb{R}^n }^2 \Phi_j ( \| \Delta \vecf \|_{ \mathbb{R}^n } ) }
	\\
	& \quad \quad
	- \,
	\left\{
	\mathscr{G}_{ \Phi , j } ( \vecf ) [ \vecpsi ]
	-
	\mathscr{M}_{ \Phi , j } ( \vecf )
	\left(
	\vectau_1 \cdot \vecpsi_1^\prime
	+ 
	\vectau_2 \cdot \vecpsi_2^\prime
	\right)
	\right\}
	\Xi_j ( \| \Delta \vecf \|_{ \mathbb{R}^n } )
	\frac { \Delta \vecf \cdot \Delta \vecphi } { \| \Delta \vecf \|_{ \mathbb{R}^n }^2 }
	\\
	& \quad \quad
	- \,
	\mathscr{M}_{ \Phi , j } ( \vecf )
	\frac { \Xi_j^\prime ( \| \Delta \vecf \|_{ \mathbb{R}^n } ) \Delta \vecf \cdot \Delta \vecpsi }
	{ \| \Delta \vecf \|_{ \mathbb{R}^n } }
	\frac { \Delta \vecf \cdot \Delta \vecphi } { \| \Delta \vecf \|_{ \mathbb{R}^n }^2 }
	\\
	& \quad \quad
	- \,
	\mathscr{M}_{ \Phi , j } ( \vecf ) \Xi_j ( \| \Delta \vecf \|_{ \mathbb{R}^n } )
	\left\{
	\frac { \Delta \vecphi \cdot \Delta \vecpsi } { \| \Delta \vecf \|_{ \mathbb{R}^n }^2 }
	-
	\frac { 2 ( \Delta \vecf \cdot \Delta \vecphi ) ( \Delta \vecf \cdot \Delta \vecpsi ) }
	{ \| \Delta \vecf \|_{ \mathbb{R}^n }^4 }
	\right\}
	\\
	& \quad
	=
	-
	\mathscr{G}_{ \Phi , j } ( \vecf ) [ \vecphi ]
	\left( \vectau_1 \cdot \vecpsi_1^\prime + \vectau_2 \cdot \vecpsi_2^\prime \right)
	\\
	& \quad \quad
	+ \,
	\frac {
	Q_{ j,1 } \vecphi \cdot Q_{j,2} \vecpsi
	+
	Q_{ j,2 } \vecphi \cdot Q_{j,1} \vecpsi
	-
	\left(
	S_{ j,1 } ( \vecf , \vecphi ) S_{ j,2 } ( \vecf , \vecpsi )
	+
	S_{ j,2 } ( \vecf , \vecphi ) S_{ j,1 } ( \vecf , \vecpsi )
	\right)
	}
	{ \Phi_j ( \| \Delta \vecf \|_{ \mathbb{R}^n } ) }
	\\
	& \quad \quad
	- \,
	\frac {
	\mathscr{P}_j ( \vecf ) [ \vecphi ]
	\Xi_j ( \| \Delta \vecf \|_{ \mathbb{R}^n } ) ( \Delta \vecf \cdot \Delta \vecpsi )
	}
	{ \| \Delta \vecf \|_{ \mathbb{R}^n }^2 \Phi_j ( \| \Delta \vecf \|_{ \mathbb{R}^n } ) }
	\\
	& \quad \quad
	- \,
	\mathscr{G}_{ \Phi , j } ( \vecf ) [ \vecpsi ]
	\Xi_j ( \| \Delta \vecf \|_{ \mathbb{R}^n } )
	\frac { \Delta \vecf \cdot \Delta \vecphi } { \| \Delta \vecf \|_{ \mathbb{R}^n }^2 }
	\\
	& \quad \quad
	- \,
	\mathscr{M}_{ \Phi , j } ( \vecf )
	\frac { \Xi_j^\prime ( \| \Delta \vecf \|_{ \mathbb{R}^n } ) \Delta \vecf \cdot \Delta \vecpsi }
	{ \| \Delta \vecf \|_{ \mathbb{R}^n } }
	\frac { \Delta \vecf \cdot \Delta \vecphi } { \| \Delta \vecf \|_{ \mathbb{R}^n }^2 }
	\\
	& \quad \quad
	- \,
	\mathscr{M}_{ \Phi , j } ( \vecf ) \Xi_j ( \| \Delta \vecf \|_{ \mathbb{R}^n } )
	\left\{
	\frac { \Delta \vecphi \cdot \Delta \vecpsi } { \| \Delta \vecf \|_{ \mathbb{R}^n }^2 }
	-
	\frac { 2 ( \Delta \vecf \cdot \Delta \vecphi ) ( \Delta \vecf \cdot \Delta \vecpsi ) }
	{ \| \Delta \vecf \|_{ \mathbb{R}^n }^4 }
	\right\}
	\\
	& \quad
	=
	-
	\mathscr{G}_{ \Phi , j } ( \vecf ) [ \vecphi ]
	\left( \vectau_1 \cdot \vecpsi_1^\prime + \vectau_2 \cdot \vecpsi_2^\prime \right)
	\\
	& \quad \quad
	+ \,
	\frac {
	Q_{ j,1 } \vecphi \cdot Q_{j,2} \vecpsi
	+
	Q_{ j,2 } \vecphi \cdot Q_{j,1} \vecpsi
	-
	\left(
	S_{ j,1 } ( \vecf , \vecphi ) S_{ j,2 } ( \vecf , \vecpsi )
	+
	S_{ j,2 } ( \vecf , \vecphi ) S_{ j,1 } ( \vecf , \vecpsi )
	\right)
	}
	{ \Phi_j ( \| \Delta \vecf \|_{ \mathbb{R}^n } ) }
	\\
	& \quad \quad
	- \,
	\left\{
	\mathscr{G}_{ \Phi , j } ( \vecf ) [ \vecphi ]
	+
	\frac { \mathscr{M}_{ \Phi , j } ( \vecf )
	\Xi_j ( \| \Delta \vecf \|_{ \mathbb{R}^n } )
	( \Delta \vecf \cdot \Delta \vecphi ) }
	{ \| \Delta \vecf \|_{ \mathbb{R}^n }^2 }
	\right\}
	\frac {
	\Xi_j ( \| \Delta \vecf \|_{ \mathbb{R}^n } ) ( \Delta \vecf \cdot \Delta \vecpsi )
	}
	{ \| \Delta \vecf \|_{ \mathbb{R}^n }^2 }
	\\
	& \quad \quad
	- \,
	\mathscr{G}_{ \Phi , j } ( \vecf ) [ \vecpsi ]
	\Xi_j ( \| \Delta \vecf \|_{ \mathbb{R}^n } )
	\frac { \Delta \vecf \cdot \Delta \vecphi } { \| \Delta \vecf \|_{ \mathbb{R}^n }^2 }
	\\
	& \quad \quad
	- \,
	\mathscr{M}_{ \Phi , j } ( \vecf )
	\| \Delta \vecf \|_{ \mathbb{R}^n } \Xi_j^\prime ( \| \Delta \vecf \|_{ \mathbb{R}^n } )
	\frac { ( \Delta \vecf \cdot \vecphi ) ( \Delta \vecf \cdot \Delta \vecpsi ) } { \| \Delta \vecf \|_{ \mathbb{R}^n }^4 }
	\\
	& \quad \quad
	- \,
	\mathscr{M}_{ \Phi , j } ( \vecf ) \Xi_j ( \| \Delta \vecf \|_{ \mathbb{R}^n } )
	\left\{
	\frac { \Delta \vecphi \cdot \Delta \vecpsi } { \| \Delta \vecf \|_{ \mathbb{R}^n }^2 }
	-
	\frac { 2 ( \Delta \vecf \cdot \Delta \vecphi ) ( \Delta \vecf \cdot \Delta \vecpsi ) }
	{ \| \Delta \vecf \|_{ \mathbb{R}^n }^4 }
	\right\}
\end{align*}
leads to our conclusion 
\begin{align*}
	&
	\mathscr{H}_{ \Phi , j } ( \vecf ) [ \vecphi , \vecpsi ]
	\\
	& \quad
	=
	\delta \left( \mathscr{G}_{ \Phi , j } ( \vecf ) [ \vecphi ] \right) [ \vecpsi ]
	+
	\mathscr{G}_{ \Phi , j } ( \vecf ) [ \vecphi ]
	\left( \vectau_1 \cdot \vecpsi_1^\prime + \vectau_2 \cdot \vecpsi_2^\prime \right)
	\\
	& \quad
	=
	\frac {
	Q_{ j,1 } \vecphi \cdot Q_{j,2} \vecpsi
	+
	Q_{ j,2 } \vecphi \cdot Q_{j,1} \vecpsi
	-
	\left(
	S_{ j,1 } ( \vecf , \vecphi ) S_{ j,2 } ( \vecf , \vecpsi )
	+
	S_{ j,2 } ( \vecf , \vecphi ) S_{ j,1 } ( \vecf , \vecpsi )
	\right)
	}
	{ \Phi_j ( \| \Delta \vecf \|_{ \mathbb{R}^n } ) }
	\\
	& \quad \quad
	- \,
	\left\{
	\mathscr{G}_{ \Phi , j } ( \vecf ) [ \vecphi ]
	+
	\frac { \mathscr{M}_{ \Phi , j } ( \vecf )
	\Xi_j ( \| \Delta \vecf \|_{ \mathbb{R}^n } )
	( \Delta \vecf \cdot \Delta \vecphi ) }
	{ \| \Delta \vecf \|_{ \mathbb{R}^n }^2 }
	\right\}
	\frac {
	\Xi_j ( \| \Delta \vecf \|_{ \mathbb{R}^n } ) ( \Delta \vecf \cdot \Delta \vecpsi )
	}
	{ \| \Delta \vecf \|_{ \mathbb{R}^n }^2 }
	\\
	& \quad \quad
	- \,
	\mathscr{G}_{ \Phi , j } ( \vecf ) [ \vecpsi ]
	\Xi_j ( \| \Delta \vecf \|_{ \mathbb{R}^n } )
	\frac { \Delta \vecf \cdot \Delta \vecphi } { \| \Delta \vecf \|_{ \mathbb{R}^n }^2 }
	\\
	& \quad \quad
	- \,
	\mathscr{M}_{ \Phi , j } ( \vecf )
	\| \Delta \vecf \|_{ \mathbb{R}^n } \Xi_j^\prime ( \| \Delta \vecf \|_{ \mathbb{R}^n } )
	\frac { ( \Delta \vecf \cdot \vecphi ) ( \Delta \vecf \cdot \Delta \vecpsi ) } { \| \Delta \vecf \|_{ \mathbb{R}^n }^4 }
	\\
	& \quad \quad
	- \,
	\mathscr{M}_{ \Phi , j } ( \vecf ) \Xi_j ( \| \Delta \vecf \|_{ \mathbb{R}^n } )
	\left\{
	\frac { \Delta \vecphi \cdot \Delta \vecpsi } { \| \Delta \vecf \|_{ \mathbb{R}^n }^2 }
	-
	\frac { 2 ( \Delta \vecf \cdot \Delta\vecphi ) ( \Delta \vecf \cdot \Delta \vecpsi ) }
	{ \| \Delta \vecf \|_{ \mathbb{R}^n }^4 }
	\right\}
	\\
	& \quad
	=
	\frac {
	Q_{ j,1 } \vecphi \cdot Q_{j,2} \vecpsi
	+
	Q_{ j,2 } \vecphi \cdot Q_{j,1} \vecpsi
	-
	\left(
	S_{ j,1 } ( \vecf , \vecphi ) S_{ j,2 } ( \vecf , \vecpsi )
	+
	S_{ j,2 } ( \vecf , \vecphi ) S_{ j,1 } ( \vecf , \vecpsi )
	\right)
	}
	{ \Phi_j ( \| \Delta \vecf \|_{ \mathbb{R}^n } ) }
	\\
	& \quad \quad
	- \,
	\mathscr{G}_{ \Phi , j } ( \vecf ) [ \vecphi ]
	\Xi_j ( \| \Delta \vecf \|_{ \mathbb{R}^n } )
	\frac { \Delta \vecf \cdot \Delta \vecpsi } { \| \Delta \vecf \|_{ \mathbb{R}^n }^2 }
	-
	\mathscr{G}_{ \Phi , j } ( \vecf ) [ \vecpsi ]
	\Xi_j ( \| \Delta \vecf \|_{ \mathbb{R}^n } )
	\frac { \Delta \vecf \cdot \Delta \vecphi } { \| \Delta \vecf \|_{ \mathbb{R}^n }^2 }
	\\
	& \quad \quad
	- \,
	\mathscr{M}_{ \Phi , j } ( \vecf ) \Xi_j ( \| \Delta \vecf \|_{ \mathbb{R}^n } )
	\frac { \Delta \vecphi \cdot \Delta\vecpsi } { \| \Delta \vecf \|_{ \mathbb{R}^n }^2 }
	\\
	& \quad \quad
	+ \,
	\mathscr{M}_{ \Phi , j } ( \vecf )
	\left(
	2 \Xi_j ( \| \Delta \vecf \|_{ \mathbb{R}^n } )
	- \Xi_j ( \| \Delta \vecf \|_{ \mathbb{R}^n } )^2
	- \| \Delta \vecf \|_{ \mathbb{R}^n } \Xi_j^\prime ( \| \Delta \vecf \|_{ \mathbb{R}^n } )
	\right)
	\frac { ( \Delta \vecf \cdot \Delta \vecphi ) ( \Delta \vecf \cdot \Delta \vecpsi ) }
	{ \| \Delta \vecf \|_{ \mathbb{R}^n }^4 }. 
\end{align*}
\qed
\end{proof}
\begin{rem}
$ \Xi_j $ is a constant when $ \Phi (x) = x^\alpha $, and hence the above equation becomes simpler. 
\end{rem}
\section{Sufficient conditions for (A.3) and (A.4)}
\label{Sufficient conditions}
\par
In this section, we consider self-repulsiveness of $ W_\Phi $ and sufficient conditions of $ \Phi $ for assumptions (A.3), (A.4). 
Set
\begin{itemize}
\item[{\rm (A.6)}]
	$ \Phi (x) = \mathcal{O} ( x^2 ) $ ($ x \to +0 $).
\end{itemize}
\begin{prop}
If conditions {\rm (A.1)} and {\rm (A.6)} hold,
then $ W_\Phi $ is self-repulsive. 
\end{prop}
\begin{proof}
We prove the assertion in the same way as \cite{OH3}. 
Here denying the existence of $ s_\ast \ne s_\dagger $ satisfying $ \vecf ( s_\ast ) = \vecf ( s_\dagger ) $,
we will show that the energy is infinite.
Since the energy density of $ \mathcal{E}_\Phi $ is non-negative, 
\[
	\mathcal{E}_\Phi ( \vecf )
	\geqq
	\iint_{ | s_1 - s_\ast |^2 + | s_2 - s_\dagger |^2 \leqq \epsilon^2 }
	\left(
	\frac 1 { \Phi ( \| \vecf ( s_1 ) - \vecf ( s_2 ) \|_{ \mathbb{R}^n } ) }
	-
	\frac 1 { \Phi ( \mathscr{D} ( \vecf ( s_1 ) , \vecf ( s_2 ) ) ) }
	\right)
	d s_1 d s_2
\]
holds for sufficiently small $ \epsilon > 0 $. 
There exists $ \delta > 0 $ which satisfies \[
	\mathscr{D} ( \vecf ( s_1 ) , \vecf ( s_2 ) ) \geqq \delta
\]
independent of $ s_1 $ and $ s_2 $ which belong to interval of integration since $ s_\ast \ne s_\dagger $. 
Hence
\[
	-
	\frac 1 { \Phi ( \mathscr{D} ( \vecf ( s_1 ) , \vecf ( s_2 ) ) ) }
	\geqq
	-
	\frac 1 { \Phi ( \delta ) }
\]
follows. Since there exists $ C > 0 $ with $ \Phi ( x ) \leqq C x^2 $ when $ x \in \left( 0 , \frac { \mathcal{L} } 2 \right] $, it holds that
\[
	\frac 1 { \Phi ( \| \vecf ( s_1 ) - \vecf ( s_2 ) \|_{ \mathbb{R}^n } ) }
	-
	\frac 1 { \Phi ( \mathscr{D} ( \vecf ( s_1 ) , \vecf ( s_2 ) ) ) }
	\geqq
	\frac 1 { C \| \vecf ( s_1 ) - \vecf ( s_2 ) \|_{ \mathbb{R}^n }^2 }
	-
	\frac 1 { \Phi ( \delta ) }. 
\]
It is sufficient to show that
\[
	\iint_{ | s_1 - s_\ast |^2 + | s_2 - s_\dagger |^2 \leqq \epsilon^2 }
	\frac { d s_1 d s_2 } { \| \vecf ( s_1 ) - \vecf ( s_2 ) \|_{ \mathbb{R}^n }^2 }
	=
	\infty
\]
and then $ \mathcal{E}_\Phi ( \vecf ) = \infty $ contradicts to the fact. 
We denote $ s_1 - s_\ast $ and $ s_2 - s_\dagger $ simply by $ s_1 $ and $ s_2 $ respectively. 
Simple calculation 
\begin{align*}
	\| \vecf ( s_1 + s_\ast ) - \vecf ( s_2 + s_\dagger ) \|_{ \mathbb{R}^n }
	= & \
	\| \vecf ( s_1 + s_\ast ) - \vecf ( s_\ast ) + \vecf ( s_\dagger ) - \vecf ( s_2 + s_\dagger ) \|_{ \mathbb{R}^n }
	\\
	\leqq & \
	\| \vecf ( s_1 + s_\ast ) - \vecf ( s_\ast ) \|_{ \mathbb{R}^n }
	+
	\| \vecf ( s_\dagger ) - \vecf ( s_2 + s_\dagger ) \|_{ \mathbb{R}^n }
	\\
	\leqq & \
	| s_1 | + | s_2 |
\end{align*}
leads 
\[
	\iint_{ s_1^2 + s_2^2 \leqq \epsilon^2 }
	\frac { d s_1 d s_2 } { \| \vecf ( s_1 + s_\ast ) - \vecf ( s_2 + \delta ) \|_{ \mathbb{R}^n }^2 }
	\geqq
	C
	\iint_{ s_1^2 + s_2^2 \leqq \epsilon^2 }
	\frac { d s_1 d s_2 } { s_1^2 + s_2^2 }
	= \infty. 
\]
\qed
\end{proof}
\par
The following three propositions are based on Blatt's idea in \cite{Blatt}. 
\begin{prop}
We assume {\rm (A.1)} and {\rm (A.6)}. 
If $ \vecf \in W_\Phi $ and $ \mathcal{E}_\Phi ( \vecf ) < \infty $, 
$ \vecf $ is bi-Lipschitz with regard to the arc-length parameter. 
\end{prop}
\begin{proof}
Since $ \| \Delta \vecf \|_{ \mathbb{R}^n } \leqq \mathscr{D} ( \vecf ( s_1 ) , \vecf ( s_2 ) ) $ is obvious for the arc-length parameter, we show the opposite inequality. 
From $ \vecf \in W_\Phi $, we gain 
\[
	\int_{ \mathbb{R} / {\cal L} \mathbb{Z} }
	\int_{ - \frac {\cal L} 2 }^{ \frac {\cal L} 2 }
	\frac { \| \vectau ( s_1 + s_2 ) - \vectau ( s_1 ) \|_{ \mathbb{R}^n }^2 }
	{ \Phi ( | s_2 | ) }
	\, d s_2 d s_1
	<
	\infty. 
\]
For $ r \in ( 0 , \frac {\cal L} 2 ) $, 
\[
	\frac { \| \vectau ( s_1 + s_2 ) - \vectau ( s_1 ) \|_{ \mathbb{R}^n }^2 }
	{ \Phi ( | s_2 | ) }
	\chi_{ [-r,r] } ( s_2 )
	\leqq
	\frac { \| \vectau ( s_1 + s_2 ) - \vectau ( s_1 ) \|_{ \mathbb{R}^n }^2 }
	{ \Phi ( | s_2 | ) }
\]
holds and from Lebesgue's convergence theorem, we obtain 
\[
	\lim_{ r \to + 0 }
	\int_{ \mathbb{R} / {\cal L} \mathbb{Z} }
	\int_{ - r }^r
	\frac { \| \vectau ( s_1 + s_2 ) - \vectau ( s_1 ) \|_{ \mathbb{R}^n }^2 }
	{ \Phi ( | s_2 | ) }
	\, d s_2 d s_1
	=
	0. 
\]
Hence there exists $ \delta \in ( 0 , \frac { \min \{ 1 , {\cal L} \} } 2 ) $ which satisfies
\[
	\sup_{ s \in \mathbb{R} / {\cal L} \mathbb{Z} }
	\int_{ s - r }^{ s + r }
	\int_{ - r }^r
	\frac { \| \vectau ( s_1 + s_2 ) - \vectau ( s_1 ) \|_{ \mathbb{R}^n }^2 }
	{ \Phi ( | s_2 | ) }
	\, d s_2 d s_1
	\leqq
	\left( \frac 12 \right)^2 
	\inf_{ x \in \left( 0 , \frac { \mathcal{L} } 2 \right]} \frac { x^2 } { \Phi ( x ) }
\]
for $ r \leqq \delta $ and then we obtain 
\begin{align*}
	&
	\frac 1 { 2r }
	\int_{ s-r }^{ s+r }
	\left\| \vectau( s_1 )
	- \frac 1 { 2r } \int_{ s-r }^{ s+r } \vectau ( s_2 ) \, d s_2
	\right\|_{ \mathbb{R}^n }
	d s_1
	\\
	& \quad
	\leqq
	\frac 1 { 4 r^2 }
	\int_{ s-r }^{ s+r }
	\int_{ s-r }^{ s+r }
	\| \vectau ( s_1 ) - \vectau ( s_2 ) \|_{ \mathbb{R}^n }
	d s_1 d s_2
	\\
	& \quad
	\leqq
	\left(
	\frac 1 { 4 r^2 }
	\int_{ s-r }^{ s+r }
	\int_{ s-r }^{ s+r }
	\| \vectau ( s_1 ) - \vectau ( s_2 ) \|_{ \mathbb{R}^n }^2
	d s_1 d s_2
	\right)^{ \frac 12 }. 
\end{align*}
Consequently we obtain 
\begin{align*}
	&
	\left(
	\frac 1 { 4 r^2 }
	\int_{ s-r }^{ s+r }
	\int_{ s-r }^{ s+r }
	\| \vectau ( s_1 ) - \vectau ( s_2 ) \|_{ \mathbb{R}^n }^2
	d s_1 d s_2
	\right)^{ \frac 12 }
	\\
	& \quad
	\leqq
	\left\{
	\frac { \Phi ( 2r ) } { 4 r^2 }
	\int_{ s-r }^{ s+r }
	\int_{ s-r }^{ s+r }
	\frac { \| \vectau ( s_1 ) - \vectau ( s_2 ) \|_{ \mathbb{R}^n }^2 }
	{ \Phi ( | s_1 - s_2 | ) }
	d s_1 d s_2
	\right\}^{ \frac 12 }
	\\
	& \quad
	\leqq
	\left(
	\sup_{ x \in \left( 0 , \frac { \mathcal{L} } 2 \right]} \frac { \Phi ( x ) }  { x^2 }
	\right)^{ \frac 12 }
	\left\{
	\int_{ s-r }^{ s+r }
	\int_{ s-r }^{ s+r }
	\frac { \| \vectau ( s_1 ) - \vectau ( s_2 ) \|_{ \mathbb{R}^n }^2 }
	{ \Phi ( | s_1 - s_2 | ) }
	d s_1 d s_2
	\right\}^{ \frac 12 }
	\\
	& \quad
	\leqq
	\frac 12, 
\end{align*}
which is easily lead from 
\[
	1 \leqq \frac { \Phi ( 2r ) } { \Phi ( | s_1 - s_2 | ) }. 
\]
Since $ \displaystyle{ \left\| \frac 1 {2r} \int_{ s - r }^{ s + r } \vectau ( s_2 ) \, d s_2 \right\|_{ \mathbb{R}^n } \leqq 1 } $ holds, we have
\[
	\inf_{ \| {\scriptsize \vecv} \|_{ \mathbb{R}^n } \leqq 1 }
	\frac 1 { 2r }
	\int_{ s - r }^{ s + r } \| \vectau ( s_1 ) - \vecv \|_{ \mathbb{R}^n }
	d s_1
	\leqq \frac 12. 
\]
\par
We consider the case when $ s_j \in \mathbb{R} / {\cal L} \mathbb{Z} $ satisfies $ | s_1 - s_2 | = 2 r \leqq 2 \delta $,
and suppose that $ s_3 \in \mathbb{R} / {\cal L} \mathbb{Z} $ satisfies $ \mathscr{D} ( \vecf ( s_1 ) , \vecf ( s_3 ) ) = \mathscr{D} ( \vecf ( s_3 ) , \vecf ( s_2 ) ) = r $. 
Then it holds that 
\[
	\begin{array}{rl}
	\| \vecf ( s_1 ) - \vecf ( s_2 ) \|_{ \mathbb{R}^n }
	= & \zume
	\displaystyle{
	\left\| \int_{ s_3 - r }^{ s_3 + r } \vectau ( s ) \, ds \right\|_{ \mathbb{R}^n }
	=
	\sup_{ \| {\scriptsize \vecv} \|_{ \mathbb{R}^n } \leqq 1 }
	\int_{ s_3 - r }^{ s_3 + r }
	\vectau ( s ) \cdot \vecv \, ds
	}
	\\
	= & \zume
	\displaystyle{
	\sup_{ \| {\scriptsize \vecv} \|_{ \mathbb{R}^n } \leqq 1 }
	\int_{ s_3 - r }^{ s_3 + r }
	\vectau ( s ) \cdot
	\{ \vectau ( s ) + ( \vecv - \vectau (s) )\} \, ds
	}
	\\
	\geqq & \zume
	\displaystyle{
	\left(
	1 -
	\inf_{ \| {\scriptsize \vecv} \|_{ \mathbb{R}^n } \leqq 1 }
	\frac 1 { 2r }
	\int_{ s_3 - r }^{ s_3 + r }
	\| \vectau ( s ) - \vecv \|_{ \mathbb{R}^n }
	ds
	\right)
	| s_1 - s_2 |
	}
	\\
	\geqq & \zume
	\displaystyle{
	\frac 12 | s_1 - s_2 |
	}. 
	\end{array}
\]
\par
Set $ I_\delta = \{ ( s_1 , s_2 ) \in \mathbb{R} / {\cal L} \mathbb{Z} \times \mathbb{R} / {\cal L} \mathbb{Z} \, | \, \mathscr{D} ( \vecf ( s_1 ) , \vecf ( s_2 ) ) \geqq 2 \delta \} $. 
Since $ \vecf $ is a continuous closed curve from the previous proposition, it follows that 
\[
	\inf_{ ( s_1 , s_2 ) \in I_\delta }
	\frac { \| \vecf ( s_1 ) - \vecf ( s_2 ) \|_{ \mathbb{R}^n } } { \mathscr{D} ( \vecf ( s_1 ) , \vecf ( s_2 ) ) }
	> 0. 
\]
\qed
\end{proof}
\par
We add the following conditions. 
\begin{itemize}
\item[{\rm (A.7)}]
	\begin{itemize}
	\item[$ \bullet $]
		$ \Phi \in C^1 ( 0 , \infty ) $,
		$ \Phi ( \sqrt x ) $ are convex. 
	\item[$ \bullet $]
		There exists $ C(\mathcal{L}) > 0 $ with $ \displaystyle{ \int_{ x \in \left( 0 , \frac { \mathcal{L} } 2 \right] } \frac { x \Phi^\prime (x) } { \Phi (x) } \geqq C(\mathcal{L}) } $. 
	\item[$ \bullet $]
		There exists a positive constant $ C(\mathcal{L}) $ and a function $ \chi $ which satisfy $ \displaystyle{ \sup_{ x \in\left( 0 , \frac { \mathcal{L} } 2 \right] } \frac { \Phi (x) } { t \Phi ( t^{-1} x ) } \leqq C(\mathcal{L}) \chi (t) } $ and $ \displaystyle{ \int_0^\epsilon \chi (t) \, dt = o ( \epsilon ) } $ as $ \epsilon \to + 0 $. 
	\item[$ \bullet $]
		$ C^\infty ( \mathbb{R} / \mathcal{L} \mathbb{Z} ) $ is dense in $ W_\Phi $. 
	\end{itemize}
\end{itemize}
\begin{prop}
We assume {\rm (A.1)} and {\rm (A.7)}. 
If $ \mathcal{E}_\Phi ( \vecf ) < \infty $, then $ \vecf \in W_\Phi $. 
\end{prop}
\begin{proof}
It follows that 
\begin{align*}
	\mathcal{E}_\Phi ( \vecf )
	= & \
	\int_{ \mathbb{R} / \mathcal{L} \mathbb{Z} }
	\int_{ - \frac { \mathcal{L} } 2 }^{ \frac { \mathcal{L} } 2 }
	\left(
	\frac 1 { \Phi ( \| \Delta \vecf \|_{ \mathbb{R}^n } ) }
	-
	\frac 1 { \Phi ( | \Delta s | ) }
	\right)
	d s_1 d s_2
	\\
	= & \
	\int_{ \mathbb{R} / \mathcal{L} \mathbb{Z} }
	\int_{ - \frac { \mathcal{L} } 2 }^{ \frac { \mathcal{L} } 2 }
	\frac { \Phi ( | \Delta s | ) - \Phi ( \| \Delta \vecf \|_{ \mathbb{R}^n } ) } { \Phi ( \| \Delta \vecf \|_{ \mathbb{R}^n } ) \Phi ( | \Delta s | ) }
	\, d s_1 d s_2. 
\end{align*}
Since $ \Phi $ is in $ C^1 $,
and since $ \Phi ( \sqrt x ) $ is convex,
it follows that 
\begin{align*}
	\Phi ( | \Delta s | ) - \Phi ( \| \Delta \vecf \|_{ \mathbb{R}^n } )
	= & \
	\int_{ \| \Delta \scriptsize {\vecf} \|_{ \mathbb{R}^n }^2 }^{ | \Delta s |^2 }
	\frac d { dx } \Phi ( \sqrt x ) \, dx
	\geqq
	\int_{ \| \Delta \scriptsize {\vecf} \|_{ \mathbb{R}^n }^2 }^{ | \Delta s |^2 }
	\left. \frac d { dx } \Phi ( \sqrt x ) \right|_{ x = \| \Delta \mbox{\boldmath{\scriptsize $ f $}} \|_{ \mathbb{R}^n }^2 } dx
	\\
	= & \
	\frac { \Phi^\prime ( \| \Delta \vecf \|_{ \mathbb{R}^n } ) } { 2 \| \Delta \vecf \|_{ \mathbb{R}^n } }
	\left( | \Delta s |^2 - \| \Delta \vecf \|_{ \mathbb{R}^n }^2 \right)
	\\
	= & \
	\frac { \Phi^\prime ( \| \Delta \vecf \|_{ \mathbb{R}^n } ) } { 4 \| \Delta \vecf \|_{ \mathbb{R}^n } }
	\int_{ s_1 }^{ s_2 } \int_{ s_1 }^{ s_2 }
	\| \vectau ( s_3 ) - \vectau ( s_4 ) \|_{ \mathbb{R}^n }^2 d s_3 d s_4. 
\end{align*}
Consequently, we obtain 
\[
	\mathcal{E} ( \vecf )
	\geqq
	\frac 14
	\int_{ \mathbb{R} / \mathcal{L} \mathbb{Z} }
	\int_{ - \frac { \mathcal{L} } 2 }^{ \frac { \mathcal{L} } 2 }
	\frac { \Phi^\prime ( \| \Delta \vecf \|_{ \mathbb{R}^n } ) } { \| \Delta \vecf \|_{ \mathbb{R}^n } \Phi ( \| \Delta \vecf \|_{ \mathbb{R}^n } ) \Phi ( | \Delta s | ) }
	\int_{ s_1 }^{ s_2 } \int_{ s_1 }^{ s_2 }
	\| \vectau ( s_3 ) - \vectau ( s_4 ) \|_{ \mathbb{R}^n }^2 d s_3 d s_4
	d s_1 d s_2. 
\]
Let $ \epsilon \in \left( 0 , \frac { \mathcal{L} } 2 \right) $. 
Decomposing the interval of integration, we gain
\begin{align*}
	&
	\iint_{ ( \mathbb{R} / \mathcal{L} \mathbb{Z} )^2 }
	\frac { \| \Delta \vectau \|_{ \mathbb{R}^n }^2 } { \Phi ( | \Delta s | ) }
	\, d s_1 d s_2
	\\
	& \quad
	=
	\int_{ \mathbb{R} / \mathcal{L} \mathbb{Z} }
	\left( \int_{ s_2 - \frac { \mathcal{L} } 2 }^{ s_2 - \epsilon }
	+ \int_{ s_2 + \epsilon }^{ s_2 + \frac { \mathcal{L} } 2 } \right)
	\frac { \| \Delta \vectau \|_{ \mathbb{R}^n }^2 } { \Phi ( | \Delta s | ) }
	\, d s_1 d s_2
	+
	\int_{ \mathbb{R} / \mathcal{L} \mathbb{Z} }
	\int_{ s_2 - \epsilon }^{ s_2 + \epsilon }
	\frac { \| \Delta \vectau \|_{ \mathbb{R}^n }^2 } { \Phi ( | \Delta s | ) }
	\, d s_1 d s_2
\end{align*}
and then we have
\begin{align*}
	&
	\int_{ \mathbb{R} / \mathcal{L} \mathbb{Z} }
	\left( \int_{ s_2 - \frac { \mathcal{L} } 2 }^{ s_2 - \epsilon }
	+ \int_{ s_2 + \epsilon }^{ s_2 + \frac { \mathcal{L} } 2 } \right)
	\frac { \| \Delta \vectau \|_{ \mathbb{R}^n }^2 } { \Phi ( | \Delta s | ) }
	\, d s_1 d s_2
	\\
	& \quad
	\leqq
	\frac 2 { \Phi ( \epsilon ) }
	\iint_{ ( \mathbb{R} / \mathcal{L} \mathbb{Z} )^2 }
	\left( \| \vectau ( s_1 ) \|_{ \mathbb{R}^n }^2 + \| \vectau ( s_2 ) \|_{ \mathbb{R}^n }^2 \right)
	d s_1 d s_2
	=
	\frac { 4 \mathcal{L}^2 } { \Phi ( \epsilon ) }. 
\end{align*}
We deform like
\begin{align*}
	&
	\int_{ \mathbb{R} / \mathcal{L} \mathbb{Z} }
	\int_{ s_2 - \epsilon }^{ s_2 + \epsilon }
	\frac { \| \Delta \vectau \|_{ \mathbb{R}^n }^2 } { \Phi ( | \Delta s | ) }
	\, d s_1 d s_2
	\\
	& \quad
	\leqq
	\int_{ \mathbb{R} / \mathcal{L} \mathbb{Z} }
	\int_{ s_2 - \epsilon }^{ s_2 + \epsilon }
	\frac { \| \vectau ( s_1 ) - \vectau ( s_2 ) \|_{ \mathbb{R}^n }^2 } { \Phi ( | \Delta s | ) }
	\, d s_1 d s_2
	\\
	& \quad
	\leqq
	\frac C { \epsilon^2 }
	\int_{ \mathbb{R} / \mathcal{L} \mathbb{Z} }
	\int_{ s_2 - \epsilon }^{ s_2 + \epsilon }
	\frac 1 { \Phi ( | \Delta s | ) }
	\\
	& \quad \qquad
	\times
	\int_0^{ \frac 2 { \mathcal{L} } \epsilon }
	\int_0^{ \frac 2 { \mathcal{L} } \epsilon }
	\left(
	\| \vectau ( s_1 - ( \Delta s ) t_1 ) - \vectau ( s_2 + ( \Delta s ) t_2 ) \|_{ \mathbb{R}^n }^2
	+
	\| \vectau ( s_1 ) - \vectau ( s_1 - ( \Delta s ) t_1 ) \|_{ \mathbb{R}^n }^2
	\right.
	\\
	& \quad \qquad \qquad \qquad
	\left.
	+ \,
	\| \vectau ( s_2 + ( \Delta s ) t_2 ) - \vectau ( s_2 ) \|_{ \mathbb{R}^n }^2
	\right)
	d t_1 d t_2 d s_1 d s_2. 
\end{align*}
Let $ \epsilon > 0 $ be sufficiently small. 
Changing variables as $ s_3 =  s_1 - ( \Delta s ) t_1 $ and $ s_4 = s_2 + ( \Delta s ) t_2 $, we obtain
\begin{align*}
	&
	\frac C { \epsilon^2 }
	\int_{ \mathbb{R} / \mathcal{L} \mathbb{Z} }
	\int_{ s_2 - \epsilon }^{ s_2 + \epsilon }
	\frac 1 { \Phi ( | \Delta s | ) }
	\int_0^{ \frac 2 { \mathcal{L} } \epsilon }
	\int_0^{ \frac 2 { \mathcal{L} } \epsilon }
	\| \vectau ( s_1 - ( \Delta s ) t_1 ) - \vectau ( s_2 + ( \Delta s ) t_2 ) \|_{ \mathbb{R}^n }^2
	\,
	d t_1 d t_2 d s_1 d s_2
	\\
	& \quad
	\leqq
	\frac C { \epsilon^2 }
	\int_{ \mathbb{R} / \mathcal{L} \mathbb{Z} }
	\int_{ s_2 - \epsilon }^{ s_2 + \epsilon }
	\frac 1 { | \Delta s |^2 \Phi ( | \Delta s | ) }
	\int_{ s_1 }^{ s_2 } \int_{ s_1 }^{ s_2 }
	\| \vectau ( s_3 ) - \vectau ( s_4 ) \|_{ \mathbb{R}^n }^2
	d s_3 d s_4 d s_1 d s_2
	\\
	& \quad
	\leqq
	\frac C { \epsilon^2 }
	\iint_{ ( \mathbb{R} / \mathcal{L} \mathbb{Z} )^2 }
	\frac 1 { \| \Delta \vecf \|_{ \mathbb{R}^n }^2 \Phi ( | \Delta s | ) }
	\int_{ s_1 }^{ s_2 } \int_{ s_1 }^{ s_2 }
	\| \vectau ( s_3 ) - \vectau ( s_4 ) \|_{ \mathbb{R}^n }^2
	d s_3 d s_4 d s_1 d s_2. 
\end{align*}
A simple inequality 
\[
	\inf_{ x \in \left( 0 , \frac { \mathcal{L} } 2 \right] }
	\frac { x \Phi^\prime (x) } { \Phi (x) }
	\geqq
	C( \mathcal{L} ) > 0
\]
leads 
\begin{align*}
	&
	\frac C { \epsilon^2 }
	\iint_{ ( \mathbb{R} / \mathcal{L} \mathbb{Z} )^2 }
	\frac 1 { \| \Delta \vecf \|_{ \mathbb{R}^n }^2 \Phi ( | \Delta s | ) }
	\int_{ s_1 }^{ s_2 } \int_{ s_1 }^{ s_2 }
	\| \vectau ( s_3 ) - \vectau ( s_4 ) \|_{ \mathbb{R}^n }^2
	d s_3 d s_4 d s_1 d s_2
	\\
	& \quad
	\leqq
	\frac { C ( \mathcal{L} ) } { \epsilon^2 }
	\iint_{ ( \mathbb{R} / \mathcal{L} \mathbb{Z} )^2 }
	\frac { \Phi^\prime ( \| \Delta \vecf \|_{ \mathbb{R}^n } ) } { \| \Delta \vecf \|_{ \mathbb{R}^n } \Phi ( \| \Delta \vecf \|_{ \mathbb{R}^n } ) \Phi ( | \Delta s | )  }
	\int_{ s_1 }^{ s_2 } \int_{ s_1 }^{ s_2 }
	\| \vectau ( s_3 ) - \vectau ( s_4 ) \|_{ \mathbb{R}^n }^2
	d s_3 d s_4 d s_1 d s_2
	\\
	& \quad
	\leqq
	\frac { C ( \mathcal{L} ) } { \epsilon^2 }
	\mathcal{E}_\Phi ( \vecf )
\end{align*}
and then it holds that
\begin{align*}
	&
	\frac C { \epsilon^2 }
	\int_{ \mathbb{R} / \mathcal{L} \mathbb{Z} }
	\int_{ s_2 - \epsilon }^{ s_2 + \epsilon }
	\frac 1 { \Phi ( | \Delta s | ) }
	\int_0^{ \frac 2 { \mathcal{L} } \epsilon }
	\int_0^{ \frac 2 { \mathcal{L} } \epsilon }
	\| \vectau ( s_1 ) - \vectau ( s_1 - ( \Delta s ) t_1 ) \|_{ \mathbb{R}^n }^2
	d t_1 d t_2 d s_1 d s_2
	\\
	& \quad
	\leqq
	\frac { C ( \mathcal{L} ) } \epsilon
	\iint_{ ( \mathbb{R} / \mathcal{L} \mathbb{Z} )^2 }
	\int_0^{ \frac 2 { \mathcal{L} } \epsilon }
	\frac 1 { \Phi ( | \Delta s | ) }
	\| \vectau ( s_1 ) - \vectau ( s_1 - ( \Delta s ) t_1 ) \|_{ \mathbb{R}^n }^2
	d t_1 d s_1 d s_2
	\\
	& \quad
	=
	\frac { C ( \mathcal{L} ) } \epsilon
	\int_0^{ \frac 2 { \mathcal{L} } \epsilon }
	\iint_{ ( \mathbb{R} / \mathcal{L} \mathbb{Z} )^2 }
	\frac 1 { \Phi ( | \Delta s | ) }
	\| \vectau ( s_1 ) - \vectau ( s_1 - ( \Delta s ) t_1 ) \|_{ \mathbb{R}^n }^2
	d s_2 d s_1 d t_1
	\\
	& \quad
	=
	\frac { C ( \mathcal{L} ) } \epsilon
	\int_0^{ \frac 2 { \mathcal{L} } \epsilon }
	\int_{ \mathbb{R} / \mathcal{L} \mathbb{Z} }
	\int_{ s_1 - \frac { \mathcal{L} } 2 }^{ s_1 + \frac { \mathcal{L} } 2 }
	\frac { \Phi ( t_1 | \Delta s | ) } { \Phi ( | \Delta s | ) }
	\frac
	{ \| \vectau ( s_1 ) - \vectau ( s_1 - ( \Delta s ) t_1 ) \|_{ \mathbb{R}^n }^2 }
	{ \Phi ( t_1 | \Delta s | ) }
	d s_2 d s_1 d t_1
	\\
	& \quad
	=
	\frac { C ( \mathcal{L} ) } \epsilon
	\int_0^{ \frac 2 { \mathcal{L} } \epsilon }
	\int_{ \mathbb{R} / \mathcal{L} \mathbb{Z} }
	\int_{ - \frac { \mathcal{L} } 2 }^{ \frac { \mathcal{L} } 2 }
	\frac { \Phi ( t_1 | s_5 | ) } { \Phi ( | s_5 | ) }
	\frac
	{ \| \vectau ( s_1 ) - \vectau ( s_1 - s_5 t_1 ) \|_{ \mathbb{R}^n }^2 }
	{ \Phi ( t_1 | s_5 | ) }
	d s_5 d s_1 d t_1
	\\
	& \quad
	=
	\frac { C ( \mathcal{L} ) } \epsilon
	\int_0^{ \frac 2 { \mathcal{L} } \epsilon }
	\int_{ \mathbb{R} / \mathcal{L} \mathbb{Z} }
	\int_{ - \frac { \mathcal{L} } 2 t_1 }^{ \frac { \mathcal{L} } 2 t_1 }
	\frac { \Phi ( | u_1 | ) } { t_1 \Phi ( t_1^{-1} | u_1 | ) }
	\frac
	{ \| \vectau ( s_1 ) - \vectau ( s_1 - u_1 ) \|_{ \mathbb{R}^n }^2 }
	{ \Phi ( | u_1 | ) }
	d u_1 d s_1 d t_1
	\\
	& \quad
	\leqq
	\frac { C ( \mathcal{L} ) } \epsilon
	\int_0^{ \frac 2 { \mathcal{L} } \epsilon }
	\int_{ \mathbb{R} / \mathcal{L} \mathbb{Z} }
	\int_{ - \epsilon }^\epsilon
	\frac { \Phi ( | u_1 | ) } { t_1 \Phi ( t_1^{-1} | u_1 | ) }
	\frac
	{ \| \vectau ( s_1 ) - \vectau ( s_1 - u_1 ) \|_{ \mathbb{R}^n }^2 }
	{ \Phi ( | u_1 | ) }
	d u_1 d s_1 d t_1. 
\end{align*}
Since there exists $ \chi $ with
\[
	\sup_{ x \in \left( 0 , \frac { \mathcal{L} } 2 \right] }
	\frac { \Phi (x) } { t \Phi ( t^{-1} x ) }
	\leqq
	C( \mathcal{L} ) \chi (t) < \infty ,
	\quad
	\int_0^\epsilon \chi (t) \, dt = o ( \epsilon )
	\quad ( \epsilon \to + 0 )
\]
for sufficiently small $ t > 0 $, it holds that 
\begin{align*}
	&
	\frac { C ( \mathcal{L} ) } \epsilon
	\int_0^{ \frac 2 { \mathcal{L} } \epsilon }
	\int_{ \mathbb{R} / \mathcal{L} \mathbb{Z} }
	\int_{ - \epsilon }^\epsilon
	\frac { \Phi ( | u_1 | ) } { t_1 \Phi ( t_1^{-1} | u_1 | ) }
	\frac
	{ \| \vectau ( s_1 ) - \vectau ( s_1 - u_1 ) \|_{ \mathbb{R}^n }^2 }
	{ \Phi ( | u_1 | ) }
	d u_1 d s_1 d t_1
	\\
	& \quad
	\leqq
	\frac { C ( \mathcal{L} ) } \epsilon
	\int_0^{ \frac 2 { \mathcal{L} } \epsilon }
	\chi ( t ) \, dt
	\int_{ \mathbb{R} / \mathcal{L} \mathbb{Z} }
	\int_{ s_2 - \epsilon }^{ s_2 + \epsilon }
	\frac
	{ \| \Delta \vectau \|_{ \mathbb{R}^n }^2 }
	{ \Phi ( | \Delta s | ) }
	d s_1 d s_2. 
\end{align*}
In a similar way, we have 
\begin{align*}
	&
	\frac C { \epsilon^2 }
	\int_{ \mathbb{R} / \mathcal{L} \mathbb{Z} } 
	\int_{ s_2 - \epsilon }^{ s_2 + \epsilon }
	\frac 1 { \Phi ( | \Delta s | ) }
	\int_0^\epsilon \int_0^\epsilon
	\| \vectau ( s_2 + ( \Delta s ) t_2 ) - \vectau ( s_2 ) \|_{ \mathbb{R}^n }^2
	d t_1 d t_2 d s_1 d s_2
	\\
	& \quad
	\leqq
	\frac { C ( \mathcal{L} ) } \epsilon
	\int_0^{ \frac 2 { \mathcal{L} } \epsilon }
	\chi ( t ) \, dt
	\int_{ \mathbb{R} / \mathcal{L} \mathbb{Z} }
	\int_{ s_2 - \epsilon }^{ s_2 + \epsilon }
	\frac
	{ \| \Delta \vectau \|_{ \mathbb{R}^n }^2 }
	{ \Phi ( | \Delta s | ) }
	d s_1 d s_2
\end{align*}
and hence 
\[
	\left( 1 -
	\frac { C ( \mathcal{L} ) } \epsilon
	\int_0^{ \frac 2 { \mathcal{L} } \epsilon }
	\chi ( t ) \, dt
	\right)
	\iint_{ ( \mathbb{R} / \mathcal{L} \mathbb{Z} )^2 }
	\frac { \| \Delta \vectau \|_{ \mathbb{R}^n }^2 } { \Phi ( | \Delta s | ) }
	\, d s_1 d s_2
	\leqq
	C ( \mathcal{L} )
	\left( \frac { \mathcal{E}_\Phi ( \vecf ) } { \epsilon^2 }
	+
	\frac 1 { \Phi ( \epsilon ) }
	\right)
\]
is showed if 
\[
	\int_{ \mathbb{R} / \mathcal{L} \mathbb{Z} }
	\int_{ s_2 - \epsilon }^{ s_2 + \epsilon }
	\frac
	{ \| \Delta \vectau \|_{ \mathbb{R}^n }^2 }
	{ \Phi ( | \Delta s | ) }
	d s_1 d s_2
	<
	\infty. 
\]
As a consequence,
if $ \vecf $ is smooth and $ \mathcal{E}_\Phi ( \vecf ) < \infty $,
then,
by taking $ \epsilon $ sufficiently small,
it holds that $ \vecf \in W_\Phi $
Since $ C^\infty ( \mathbb{R} / \mathcal{L} ) $ is assumed to be dense in $ W_\Phi $,
$ \vecf \in W_\Phi $ follows from $ \mathcal{E}_\Phi ( \vecf ) < \infty $. 
\qed
\end{proof}
\par
It is simple to show that $ \Phi (x) = x^\alpha $ {\rm ($ \alpha \in [ 2 , 3 )$)} satisfy 
{\rm (A.1)},
{\rm (A.6)} and {\rm (A.7)}, and then we have the following corollary. 
\begin{cor}
We suppose 
{\rm (A.1)},
{\rm (A.6)} and {\rm (A.7)}. 
Then we have 
{\rm (A.3)}.
In particular,
{\rm (A.3)} holds when 
$ \Phi (x) = x^\alpha $ {\rm ($ \alpha \in [ 2 , 3 )$)}. 
\end{cor}
We assume that
\begin{itemize} 
\item[{\rm (A.8)}]
	\begin{itemize}
	\item[$ \bullet $]
		$ \Phi \in C^1 ( 0 , \infty ) $,
		$ \displaystyle{ \frac 1 { \Phi ( \sqrt x ) } } $ is convex. 
	\item[$ \bullet $]
		There exists $ C( \mathcal{L} ) > 0 $ with $ \displaystyle{ \sup_{ x \in \left( 0 , \frac { \mathcal{L} } 2 \right] }
		\frac { \Phi^\prime (x) } x < C( \mathcal{L} ) } $. 
	\end{itemize}
\end{itemize}
\begin{prop}
Supposing {\rm (A.1)} and {\rm (A.8)}, we have $ \mathcal{E}_\Phi ( \vecf ) < \infty $ if $ \vecf \in W_\Phi $. 
\end{prop}
\begin{proof}
It holds that 
\begin{align*}
	\mathcal{E}_\Phi ( \vecf )
	= & \
	\iint_{ ( \mathbb{R} / \mathcal{L} \mathbb{Z} )^2 }
	\int_{ \mathscr{D} ( \mbox{\boldmath\scriptsize $ f $} ( s_1 ) , \mbox{\boldmath\scriptsize $ f $} ( s_2 ) )^2 }^{ \| \Delta \mbox{\boldmath\scriptsize $ f $} \|_{ \mathbb{R}^n }^2 }
	\frac d { dx } \frac 1 { \Phi ( \sqrt x ) }
	\, dx d s_1 d s_2
	\\
	= & \
	\iint_{ ( \mathbb{R} / \mathcal{L} \mathbb{Z} )^2 }
	\int_{ \| \Delta \mbox{\boldmath\scriptsize $ f $} \|_{ \mathbb{R}^n }^2 }^{ \mathscr{D} ( \mbox{\boldmath\scriptsize $ f $} ( s_1 ) , \mbox{\boldmath\scriptsize $ f $} ( s_2 ) )^2 }
	\left( - \frac d { dx } \frac 1 { \Phi ( \sqrt x ) } \right)
	\, dx d s_1 d s_2. 
\end{align*}
From the assumption of $ \Phi $, 
\[
	\left. \frac d { dx } \frac 1 { \Phi ( \sqrt x ) } \right|_{ x = \| \Delta \mbox{\boldmath\scriptsize $ f $} \|_{ \mathbb{R}^n }^2 }
	\leqq
	\frac d { dx } \frac 1 { \Phi ( \sqrt x ) }
	\leqq 0
\]
holds when $ x \in \left[ \| \Delta \vecf \|_{ \mathbb{R}^n }^2 , \mathscr{D} ( \vecf ( s_1 ) , \vecf ( s_2 ) )^2 \right] $ and 
\begin{align*}
	\mathcal{E}_\Phi ( \vecf )
	\leqq & \
	\iint_{ ( \mathbb{R} / \mathcal{L} \mathbb{Z} )^2 }
	\int_{ \| \Delta \mbox{\boldmath\scriptsize $ f $} \|_{ \mathbb{R}^n }^2 }^{ \mathscr{D} ( \mbox{\boldmath\scriptsize $ f $} ( s_1 ) , \mbox{\boldmath\scriptsize $ f $} ( s_2 ) )^2 }
	\left. \left( - \frac d { dx } \frac 1 { \Phi ( \sqrt x ) } \right) \right|_{ x = \| \Delta \mbox{\boldmath\scriptsize $ f $} \|_{ \mathbb{R}^n }^2 }
	\, dx d s_1 d s_2
	\\
	= & \
	\frac 12
	\iint_{ ( \mathbb{R} / \mathcal{L} \mathbb{Z} )^2 }
	\frac { \Phi^\prime ( \| \Delta \vecf \|_{ \mathbb{R}^n } ) }
	{ \| \Delta \vecf \|_{ \mathbb{R}^n } \Phi ( \| \Delta \vecf \|_{ \mathbb{R}^n } ) }
	\left( { \mathscr{D} ( \vecf ( s_1 ) , \vecf ( s_2 ) )^2 } - \| \Delta \vecf \|_{ \mathbb{R}^n }^2 \right)
	d s_1 d s_2
	\\
	= & \
	\frac 14
	\iint_{ ( \mathbb{R} / \mathcal{L} \mathbb{Z} )^2 }
	\frac { \Phi^\prime ( \| \Delta \vecf \|_{ \mathbb{R}^n } ) }
	{ \| \Delta \vecf \|_{ \mathbb{R}^n } \Phi ( \| \Delta \vecf \|_{ \mathbb{R}^n } ) }
	\left(
	\int_{ s_1 }^{ s_2 } \int_{ s_1 }^{ s_2 }
	\| \vectau ( s_3 ) - \vectau ( s_4 ) \|_{ \mathbb{R}^n }^2
	d s_3 d s_4
	\right)
	d s_1 d s_2
\end{align*}
holds. 
From $ | s_3 - s_4 | \leqq | \Delta s | $, we have
\[
	\mathcal{E}_\Phi ( \vecf )
	\leqq
	\frac 14
	\iint_{ ( \mathbb{R} / \mathcal{L} \mathbb{Z} )^2 }
	\frac { \Phi^\prime ( \| \Delta \vecf \|_{ \mathbb{R}^n } ) }
	{ \| \Delta \vecf \|_{ \mathbb{R}^n } }
	\left(
	\int_{ s_1 }^{ s_2 } \int_{ s_1 }^{ s_2 }
	\frac { \| \vectau ( s_3 ) - \vectau ( s_4 ) \|_{ \mathbb{R}^n }^2 } { \Phi ( | s_3 - s_4 | ) }
	d s_3 d s_4
	\right)
	d s_1 d s_2. 
\]
If 
\[
	\sup_{ x \in \left( 0 , \frac { \mathcal{L} } 2 \right] }
	\frac { \Phi^\prime (x) } x
	\leqq
	C( \mathcal{L} )
\]
holds, then it is showed that 
\[
	\mathcal{E}_\Phi ( \vecf )
	\leqq
	C( \mathcal{L} ) \mathcal{L}^2
	\iint_{ ( \mathbb{R} / \mathcal{L} \mathbb{Z} )^2 }
	\frac { \| \vectau ( s_3 ) - \vectau ( s_4 ) \|_{ \mathbb{R}^n }^2 } { \Phi ( | s_3 - s_4 | ) }
	\, d s_3 d s_4. 
\]
\qed
\end{proof}
\par
Lastly, we give a sufficient condition of $ \Phi $ for (A.4). 
\begin{itemize}
\item[{\rm (A.9)}]
	For any $ \lambda \in ( 0,1 ) $, we assume that
	\[
		\limsup_{ \epsilon \to + 0 }
		\epsilon^2
		\sup_{ x \in [ \lambda^2 \epsilon^2 , \epsilon^2 ] }
		\left| \frac d { dx } \Lambda ( \sqrt x ) \right|
		\sup_{ y \in ( 0 , \epsilon ] } \Phi (y) < \infty .
	\]
\end{itemize}
\begin{prop}
We suppose that 
{\rm (A.1)},
{\rm (A.2)} and {\rm (A.9)}. 
If $ \vecf \in W_\Phi $ and $ \vecf $ is bi-Lipschitz,
then $ (\ast) $ of {\rm (A.4)} holds. 
\end{prop}
\begin{proof}
Since 
\[
	\Lambda (x) - \Lambda (y)
	=
	\int_{ y^2 }^{ x^2 } \frac d { dt } \Lambda ( \sqrt t ) \, dt
\]
follows for $ x \geqq y > 0 $, it holds that 
\[
	\left| \Lambda (x) - \Lambda (y) \right|
	=
	\sup_{ t \in [ y^2 , x^2 ] } \left| \frac d { dt } \Lambda ( \sqrt x ) \right| \left| x^2 - y^2 \right|. 
\]
From the bi-Lipschitz estimate, there exists a positive constant $ \lambda \in ( 0,1 ) $ independent  of $ s_1 $ and $ s_2 $ and it holds that
\[
	\lambda \mathscr{D} ( \vecf ( s_1 ) , \vecf ( s_2 ) )
	\leqq
	\| \Delta \vecf \|_{ \mathbb{R}^n }
	\leqq
	\mathscr{D} ( \vecf ( s_1 ) , \vecf ( s_2 ) ). 
\]
From 
\[
	\lambda^2 \epsilon^2
	\leqq
	\| \vecf ( s_1 + \epsilon ) - \vecf ( s_1 ) \|_{ \mathbb{R}^n }^2
	\leqq
	\epsilon^2, 
\]
\begin{align*}
	&
	\left| \Lambda ( \| \vecf ( s_1 + \epsilon ) - \vecf ( s_1 ) \|_{ \mathbb{R}^n } ) - \Lambda ( \epsilon ) \right|
	\\
	& \quad
	\leqq
	\sup_{ x \in [ \lambda^2 \epsilon^2 , \epsilon^2 ] } \left| \frac d { dx } \Lambda ( \sqrt x ) \right|
	\left( \epsilon^2 -  \| \vecf ( s_1 + \epsilon ) - \vecf ( s_1 ) \|_{ \mathbb{R}^n }^2 \right)
	\\
	& \quad
	=
	\sup_{ x \in [ \lambda^2 \epsilon^2 , \epsilon^2 ] } \left| \frac d { dx } \Lambda ( \sqrt x ) \right|
	\int_{ s_1 }^{ s_1 + \epsilon }
	\int_{ s_1 }^{ s_1 + \epsilon }
	\left( 1 - \vectau ( s_3 ) \cdot \vectau ( s_4 ) \right) d s_3 d s_4
	\\
	& \quad
	=
	\frac 12
	\sup_{ x \in [ \lambda^2 \epsilon^2 , \epsilon^2 ] } \left| \frac d { dx } \Lambda ( \sqrt x ) \right|
	\int_{ s_1 }^{ s_1 + \epsilon }
	\int_{ s_1 }^{ s_1 + \epsilon }
	\left\| \vectau ( s_3 ) - \vectau ( s_4 ) \right\|_{ \mathbb{R}^n }^2 d s_3 d s_4
\end{align*}
holds and hence we have 
\begin{align*}
	&
	\epsilon
	\int_{ \mathbb{R} / \mathcal{L} \mathbb{Z} }
	\left(
	\Lambda ( \| \vecf ( s_1 ) - \vecf ( s_1 + \epsilon ) \|_{ \mathbb{R}^n } )
	-
	\Lambda ( \epsilon )
	\right)
	d s_1
	\\
	\leqq & \
	\frac \epsilon 2
	\sup_{ x \in [ \lambda^2 \epsilon^2 , \epsilon^2 ] } \left| \frac d { dx } \Lambda ( \sqrt x ) \right|
	\int_{ \mathbb{R} / \mathcal{L} \mathbb{Z} }
	\int_{ s_1 }^{ s_1 + \epsilon }
	\int_{ s_1 }^{ s_1 + \epsilon }
	\left\| \vectau ( s_3 ) - \vectau ( s_4 ) \right\|_{ \mathbb{R}^n }^2 d s_3 d s_4 d s_1. 
\end{align*}
Changing order of integration, we have
\begin{align*}
	&
	\epsilon
	\int_{ \mathbb{R} / \mathcal{L} \mathbb{Z} }
	\left(
	\Lambda ( \| \vecf ( s_1 ) - \vecf ( s_1 + \epsilon ) \|_{ \mathbb{R}^n } )
	-
	\Lambda ( \epsilon )
	\right)
	d s_1
	\\
	\leqq & \
	\frac \epsilon 2
	\sup_{ x \in [ \lambda^2 \epsilon^2 , \epsilon^2 ] } \left| \frac d { dx } \Lambda ( \sqrt x ) \right|
	\int_{ \mathbb{R} / \mathcal{L} \mathbb{Z} }
	\int_{ s_4 - \epsilon }^{ s_4 + \epsilon }
	\int_{ s_3 - \epsilon  }^{ s_3 }
	\left\| \vectau ( s_3 ) - \vectau ( s_4 ) \right\|_{ \mathbb{R}^n }^2 d s_1 d s_3 d s_4
	\\
	= & \
	\frac { \epsilon^2 } 2
	\sup_{ x \in [ \lambda^2 \epsilon^2 , \epsilon^2 ] } \left| \frac d { dx } \Lambda ( \sqrt x ) \right|
	\int_{ \mathbb{R} / \mathcal{L} \mathbb{Z} }
	\int_{ s_4 - \epsilon }^{ s_4 + \epsilon }
	\left\| \vectau ( s_3 ) - \vectau ( s_4 ) \right\|_{ \mathbb{R}^n }^2 d s_3 d s_4
	\\
	\leqq & \
	\frac { \epsilon^2 } 2
	\sup_{ x \in [ \lambda^2 \epsilon^2 , \epsilon^2 ] } \left| \frac d { dx } \Lambda ( \sqrt x ) \right|
	\sup_{ y \in ( 0 , \epsilon ] } \Phi (y)
	\int_{ \mathbb{R} / \mathcal{L} \mathbb{Z} }
	\int_{ s_4 - \epsilon }^{ s_4 + \epsilon }
	\frac { \left\| \vectau ( s_3 ) - \vectau ( s_4 ) \right\|_{ \mathbb{R}^n }^2 }
	{ \Phi ( | s_3 - s_4 | ) }
	d s_3 d s_4
	\\
	\leqq & \
	C ( \lambda )
	\int_{ \mathbb{R} / \mathcal{L} \mathbb{Z} }
	\int_{ s_4 - \epsilon }^{ s_4 + \epsilon }
	\frac { \left\| \vectau ( s_3 ) - \vectau ( s_4 ) \right\|_{ \mathbb{R}^n }^2 }
	{ \Phi ( | s_3 - s_4 | ) }
	d s_3 d s_4. 
\end{align*}
From absolute integrability of integration, 
\[
	\lim_{ \epsilon \to + 0 }
	\int_{ \mathbb{R} / \mathcal{L} \mathbb{Z} }
	\int_{ s_4 - \epsilon }^{ s_4 + \epsilon }
	\frac { \left\| \vectau ( s_3 ) - \vectau ( s_4 ) \right\|_{ \mathbb{R}^n }^2 }
	{ \Phi ( | s_3 - s_4 | ) }
	d s_3 d s_4
	= 0
\]
holds if $ \vecf \in W_\Phi $. 
\qed
\end{proof}
\begin{itemize}
\item[{\rm (A.10)}]
	For any $ \lambda \in ( 0,1 ) $, we assume that 
	$ \displaystyle{
		\limsup_{ \epsilon \to + 0 }
		\sup_{ x \in ( 0 , \epsilon ] } \Phi (x)
		\sup_{ y \in ( \lambda \epsilon , \epsilon ] } | \Lambda (y) | } < \infty $. 
\end{itemize}
\begin{prop}
We suppose that {\rm (A.1)},
{\rm (A.2)} and {\rm (A.10)}. 
If $ \vecf \in W_\Phi $ and $ \vecf $ is bi-Lipschitz,
then $ (\dagger) $ of {\rm (A.4)} holds. 
\end{prop}
\begin{proof}
There exists $ \lambda \in ( 0 , 1 ] $ with $ \lambda \epsilon \leqq \| \vecf ( s_1 + \epsilon , s_1 ) - \vecf ( s_1 ) \|_{ \mathbb{R}^n } \leqq \epsilon $ from bi-Lipschitz estimate. Therefore we have 
\begin{align*}
	&
	\int_{ \mathbb{R} / \mathcal{L} \mathbb{Z} }
	| \Lambda ( \| \vecf ( s_1 + \epsilon ) - \vecf ( s_1 ) \|_{ \mathbb{R}^n } ) |
	\int_{ s_1 }^{ s_1 + \epsilon }
	\| \vectau ( s_1 ) - \vectau ( s_3 ) \|_{ \mathbb{R}^n }^2 d s_3
	d s_1
	\\
	& \quad
	\leqq
	\sup_{ x \in ( 0 , \epsilon ] } \Phi (x)
	\sup_{ y \in ( \lambda \epsilon , \epsilon ] } | \Lambda (y) |
	\int_{ \mathbb{R} / \mathcal{L} \mathbb{Z} }
	\int_{ s_1 }^{ s_1 + \epsilon }
	\frac { \| \vectau ( s_1 ) - \vectau ( s_3 ) \|_{ \mathbb{R}^n }^2 }
	{ | s_3 - s_1 | }
	d s_3
	d s_1
	\\
	& \quad
	\leqq
	C( \lambda )
	\int_{ \mathbb{R} / \mathcal{L} \mathbb{Z} }
	\int_{ s_1 }^{ s_1 + \epsilon }
	\frac { \| \vectau ( s_1 ) - \vectau ( s_3 ) \|_{ \mathbb{R}^n }^2 }
	{ | s_1 - s_3 | }
	d s_3
	d s_1. 
\end{align*}
If $ \vecf \in W_\Phi $,
then
\[
	\lim_{ \epsilon \to + 0 }
	\int_{ \mathbb{R} / \mathcal{L} \mathbb{Z} }
	\int_{ s_1 }^{ s_1 + \epsilon }
	\frac { \| \vectau ( s_1 ) - \vectau ( s_3 ) \|_{ \mathbb{R}^n }^2 }
	{ | s_1 - s_3 | }
	d s_3
	d s_1
	= 0
\]
holds from absolute continuity of integration. 
\qed
\end{proof}
\par
It is easily shown that $ \Phi (x) = x^\alpha $ {\rm ($ \alpha \in [ 2 , 3 )$)} satisfies 
{\rm (A.1)},
{\rm (A.2)},
{\rm (A.10)} and {\rm (A.11)} 
and therefore we have next corollary. 
\begin{cor}
We suppose {\rm (A.1)},
{\rm (A.2)},
{\rm (A.10)} and {\rm (A.11)},
and then {\rm (A.4)} holds. 
In particular, {\rm (A.4)} holds when 
$ \Phi (x) = x^\alpha $ {\rm ($ \alpha \in [ 2 , 3 )$)}. 
\end{cor}
\par\noindent
{\bf Acknowledgment}.
The authors express their appreciation to Professor Jun O'Hara for his information of article \cite{Brylinski}.

\end{document}